\documentclass[11pt,twoside,reqno]{amsart}
\usepackage{bbm}
\usepackage{verbatim}
\usepackage{graphicx,xcolor}
\usepackage{enumerate}
\usepackage[all]{xy}
\usepackage[hyperindex=true,plainpages=false,colorlinks=false,pdfpagelabels]{hyperref}

\newtheorem{theorem}{Theorem}[section]
\newtheorem{proposition}[theorem]{Proposition}
\newtheorem{lemma}[theorem]{Lemma}
\newtheorem{corollary}[theorem]{Corollary}

\theoremstyle{remark}
\newtheorem{remark}[theorem]{Remark}
\newtheorem{definition}[theorem]{Definition}

\numberwithin{equation}{section}
\newcommand{\wt}[1]{\widetilde{#1}}

\DeclareMathOperator{\SL}{SL}

\DeclareMathOperator{\Tr}{tr}

\newcommand{\PSL}{{\rm PSL}}
\newcommand{\dvector}[1]{{\left(\begin{matrix}#1\end{matrix}\right)}}

\newcommand{\R}{\mathbb{R}}
\newcommand{\C}{\mathbb{C}}
\newcommand{\CP}{\C P}
\newcommand{\N}{\mathbb{N}}
\newcommand{\Z}{\mathbb{Z}}
\begin{document}

\title[Holomorphic curves in compact quotients of $\mathrm{SL}(2,\C)$]{On
the existence of holomorphic curves in  compact quotients of $\mathrm{SL}(2,\C)$}

\author[I. Biswas]{Indranil Biswas}

\address{School of Mathematics, Tata Institute of 
Fundamental Research, Homi Bhabha Road, Mumbai 400005, India} 

\email{indranil@math.tifr.res.in}

\author[S. Dumitrescu]{Sorin Dumitrescu}

\address{Universit\'e C\^ote d'Azur, CNRS, LJAD, France}

\email{dumitres@unice.fr}

\author[L. Heller]{Lynn Heller}

\address{Institute of Differential Geometry, Leibniz 
Universit\"at Hannover, Welfengarten 1, 30167 Hannover} 

\email{lynn.heller@math.uni-hannover.de}

\author[S. Heller]{Sebastian Heller}

\address{Institute of Differential Geometry, Leibniz 
Universit\"at Hannover, Welfengarten 1, 30167 Hannover}

\email{seb.heller@gmail.com}



\date{\today}

\begin{abstract} 
We prove the existence of a pair $(\Sigma ,\, \Gamma)$, where $\Sigma$ is a compact Riemann surface
with $\text{genus}(\Sigma)\, \geq\, 2$, and $\Gamma\, \subset\, {\rm SL}(2, \C)$ is a
cocompact lattice, such that there is a generically injective holomorphic map
$\Sigma \, \longrightarrow\, {\rm SL}(2, \C)/\Gamma$. This gives an affirmative answer to a question raised by
 Huckleberry and Winkelmann \cite{HW} and by Ghys \cite{Gh}.

\end{abstract}

\maketitle

\tableofcontents

\section{Introduction}

Compact complex manifolds with holomorphically trivial tangent bundle are known to be 
biholomorphic to a quotient of a complex Lie group $G$ by a discrete cocompact subgroup 
$\Gamma$ \cite{Wa}. These manifolds, also known as {\it parallelizable complex manifolds}, 
are K\"ahler if and only if the Lie group $G$ is abelian (in which case the manifold is a 
compact complex torus).

Whenever $G$ is a semi-simple Lie group, and $\Gamma\, \subset\, G$ is a
cocompact lattice, a theorem of Huckleberry and Margulis
\cite{HM} says that $G / \Gamma$ does not admit any complex analytic hypersurface.
In particular, the algebraic dimension of $G / \Gamma$ is zero, meaning $G / \Gamma$
does not admit any nonconstant meromorphic function.

An important class of examples consists of compact quotients of $G\,=\,\text{SL}(2, \C)$ by cocompact 
Kleinian subgroups $\Gamma\,\subset\, \text{SL}(2, \C)$. Since $\text{PSL}(2, \C)$ is the 
group of orientation preserving isometries of the hyperbolic $3$-space $\mathbb H^3$, the 
compact quotient $\text{SL}(2, \C)/ \Gamma$ is an unramified double cover of the 
$\text{SO}(3, \R)$-bundle of oriented orthonormal frames of the compact hyperbolic 
$3$-manifold $\mathbb H^3/ \Gamma$. While the embedding of $\Gamma$ into $\text{SL}(2, \C)$ 
is known to be rigid by Mostow's Theorem, the flexibility of the complex structure of 
$\text{SL}(2, \C)/ \Gamma$ was discovered by Ghys \cite{Gh} where he showed that the 
corresponding Kuranishi space has positive dimension for all $\Gamma$ with positive first 
Betti number. The corresponding compact hyperbolic $3$-manifolds can be constructed using 
Thurston's hyperbolisation Theorem (see, for instance, \cite{Th} and \cite[Lemme 6.2]{Gh} 
for constructions of compact hyperbolic $3$-manifolds with prescribed rational cohomology 
ring). Nevertheless, much of the interplay between the geometry of the compact hyperbolic 
$3$-manifold $\mathbb H^3/ \Gamma$ and the complex structure of its oriented orthonormal 
frame bundle $\text{PSL}(2, \C)/ \Gamma$ remains to be explored.

In course of his studies of the deformation space of the complex structures of $\text{SL}(2, 
\C)/ \Gamma$ Ghys \cite{Gh} encountered a problem previously raised by Huckleberry and 
Winkelmann \cite{HW} which would generalize the Huckleberry-Margulis Theorem \cite{HM} to 
holomorphic curves: Does there exist a compact 3-manifold $\text{SL}(2, \C)/ \Gamma$ 
admitting a compact holomorphic curve of genus $g \geq 2$? Note that the case of elliptic 
curves covered by one-parameters groups in $\text{SL}(2, \C)$ are well-known to exist 
in certain quotients $\text{SL}(2, \C)/ \Gamma$.

In this paper we give an affirmative answer to this open question following a strategy due 
to Ghys, see \cite{BDHH,CDHL}. We construct on the trivial holomorphic bundle of rank two 
over a Riemann surfaces $\Sigma$ an irreducible holomorphic $\text{SL}(2, 
\C)$-connection such that the image of the corresponding monodromy homomorphism lies in a 
cocompact lattice $\Gamma$ in $\text{SL}(2, \C)$. The parallel frame of this connection 
then gives rise to a holomorphic map from $\Sigma$ into the quotient $\text{SL}(2, \C)/ 
\Gamma$ which, due to the irreducibility of the connection, does not factor through any 
elliptic curve.  A first step towards realizing Ghys' strategy was previously made in 
\cite{BDHH} where holomorphic connections with (real) Fuchsian monodromy were constructed. 
In this paper we first show that every irreducible $\text{SL}(2, \R)$-representation with 
sufficient many symmetries can be realized as the monodromy of a holomorphic $\text{SL}(2, 
\C)$-connection. Then an example of such a symmetric representation contained in a cocompact 
lattice $\Gamma$ in $\text{SL}(2, \C)$ is given. We end with an outlook on the relationship 
between holomorphic curves in SL$(2, \C)/\Gamma$ and surfaces of constant mean curvature 
$H\,=\,1$ in the hyperbolic 3-space.

\section*{Acknowledgements}

The authors would like
to thank Bertrand Deroin for many useful joint discussions on the topic, which initiated our project. SH 
would like to thank Nick Schmitt for stimulating discussions about surfaces in hyperbolic 3-manifolds.

IB is partially supported by a J. C. Bose Fellowship. SD was partially supported by the 
French government through the UCAJEDI Investments in the Future project managed by the 
National Research Agency (ANR) with the reference number ANR2152IDEX201. LH is supported by 
the DFG grant HE 7914/2-1 of the DFG priority program SPP 2026 Geometry at Infinity. SH is 
supported by the DFG grant HE 6829/3-1 of the DFG priority program SPP 2026 Geometry at 
Infinity.

\section{Statements of the main theorems and strategy of proof}

Let $\Gamma\, \subset\, \SL(2, \C)$ be a cocompact lattice and consider the compact
complex $3$-manifold $N \,=\, \SL(2, \C)/\Gamma$. This three manifold can be viewed as the double covering of the
$\mathrm{SO}(3)$-frame bundle of the corresponding hyperbolic 3-orbifold $\mathbb H^3/\Gamma$ and
it is called the unitary frame bundle.
Let $\Sigma$ be a Riemann surface of genus $g\,\geq\, 2$, fix a base point $x_0\, \in\, \Sigma$ and consider an irreducible
representation
$$
\rho_\Sigma \,\colon\, \pi_1(\Sigma,\, x_0)
\,\longrightarrow\, \Gamma \,\subset\, \SL(2,\C).$$

Following an idea of Ghys,   we aim at realizing such a representation $\rho_\Sigma$ as the monodromy representation of a holomorphic flat connection $\nabla$ on the trivial holomorphic $\C^2$-bundle over $\Sigma$. Then the corresponding parallel frame $\Psi\colon \Sigma \rightarrow \SL(2, \C),$ with $\Psi(x_0) = $ Id induces a well-defined holomorphic map $f_{\rho_\Sigma}$ from $\Sigma$ into $N$ (with monodromy representation $\rho_\Sigma$).  The map $f_{\rho_\Sigma}$ does not factor through an elliptic curve since $\rho_\Sigma$ is
irreducible. 

In the following we impose various symmetries on $\Sigma$ and $\rho_\Sigma$. Let $\Sigma$ be the covering
of $\C P^1$ of degree $g+1$ defined by the equation
\begin{equation}\label{Sigmaa}
y^{g+1}\,=\,\frac{(z-p_1)(z-p_2)}{(z-p_3)(z-p_4)},
\end{equation} 
where
\begin{equation}\label{rp}
p_1\,=\, e^{i \varphi}, \quad p_2 \,=\, -e^{-i \varphi},\quad p_3 \,=
\,-e^{i \varphi},
\quad p_4 \,=\, e^{-i \varphi}
\end{equation}
with $\varphi\,\in\, (0,\, \pi/2)$, i.e.,  $\Sigma$ is totally branched
over $p_1,\, p_2,\,p_3,\, p_4$.
The representations $\rho_\Sigma \, =\,\widehat{\rho}$ we consider are compatible with the covering, i.e., it is
induced by the monodromy representation $\rho$ of a particular rank $2$
logarithmic connection $\nabla$ (described in Section \ref{ss:logconn}) over
the 4-punctured sphere 
\begin{equation}\label{S4}
S_4\,:=\, \C P^1 \setminus\{p_1,\, \dots,\, p_4\}.
\end{equation}
Fix a base point $s_0\, \in\, S_4$, and let
$\gamma_{p _j}\, \in\, \pi_1(S_4,\, s_0)$ be the curve that goes around $p _j$ anti-clockwise  for $j= 1, ..., 4$,
such that  $\gamma_{p_4}\gamma_{p_3}\gamma_{p_2}\gamma_{p_1}\,=\,1$ (see Section \ref{subs-4 punctured}).
Let $M _j$ denote the monodromy of $\nabla$ 
along $\gamma_{p _j}$. Then $\rho$ is assumed to satisfy the following RSR condition.

\begin{definition} \label{RSR-def}
An irreducible representation 
$\rho\,\colon\,\pi_1(S_4,\, s_0)\,\longrightarrow\,\SL(2,\C)$ 
is called RSR-representation if it has the following three properties:
\begin{itemize}
\item {\bf Real}:\, $\rho$ takes values in $\SL(2,\R)$ and
$$x\,=\, \Tr(M_1M_2)\,<\,-2,\ \, y\,=\, \Tr(M_2M_3)\,<\,-2,\ \, z\,=\, \Tr(M_1M_3)\,<\,-2,$$
where $M _j\,=\, \rho(\gamma_{p _j})$.

\item {\bf Symmetric}:\, the four monodromies $M_1,\,\cdots,\,M_4$ lie  in the same conjugacy class
determined by $\text{diag}(e^{2\pi i\wt{r}},\, e^{-2\pi i\wt{r}})$
for some $\wt{r}\,\in\, (\tfrac{1}{4},\,\tfrac{1}{2})$.

\item {\bf Rectangular}:\, $\text{tr}(M_1M_3)\,=\,\text{tr}(M_2M_4)$.
\end{itemize}
\end{definition}

Every $\rho\, \in\, {\rm Hom}(\pi_1(S_4,\, s_0),\, \SL(2,\C))$
such that $\rho(\gamma_{p _j})^k\,=\, 1$ for all  $j=1,\,...,\,4,$ for some integer $k\, >\, 0$,
lifts to a representation of $\pi_1(\Sigma,\, x_0)$ for the Riemann surface in
\eqref{Sigmaa} of suitable genus $g$, see Lemma \ref{trivialoncovering} for the details.
This motivates the definition of the genus of a
RSR-representation, see also \cite[Section 3]{BDHH}.

\begin{definition} \label{RSR-genus-def}
Let $\rho$ be a RSR-representation such that one (and hence all, as they have
same conjugacy class) of $M_1,\,\cdots,\,M_4$ have order
$k\,\in\,\N.$ Then, the {\bf genus} of $\rho$ is
$k-1,$ if $k$ is odd, and it is $\tfrac{k}{2}-1$ if $k$ is even.
\end{definition}

Since the symmetries are analogous to those of the 
Lawson minimal surfaces of genus $g$, we define:

\begin{definition} \label{RL}
Any homomorphism $\widehat\rho\, :\, \pi_1(\Sigma,\, x_0)\, \longrightarrow\,
\SL(2,\R)$ induced by some RSR-representation $\rho$ (defined in \eqref{Sigmaa}) 
is called {\bf real Lawson-symmetric} ({\bf RL}).
\end{definition}

Our main theorem is the following.

\begin{theorem}[{Main Theorem}]\label{maintheorem}
Let  $\rho$ be  a RSR-representation of genus $g$, and let $\Gamma$ be a compatible
cocompact lattice in ${\rm SL}(2, \C)$, meaning
$\Gamma$ contains the image of the corresponding RL-representation $\widehat\rho$.
Then there exists a genus g Riemann surface $\Sigma$
of the form \eqref{Sigmaa} and a holomorphic 
map from $\Sigma$ to the compact complex 3-manifold SL$(2, \C)/\Gamma$.
Moreover, the holomorphic map does not factor through a torus.
\end{theorem}

\begin{remark} \label{remark-infty}
The techniques presented in this paper produce in fact infinitely many holomorphic maps from Riemann surfaces $\Sigma^n,$ $n \in \N,$ of the same genus into $\SL(2, \C)/\Gamma$ inducing the same RL-representation $\widehat\rho.$ To enhance clarity we discuss here only the simplest case arising from grafting once. 
\end{remark}

\begin{remark}
The assumption of the theorem is for example fulfilled if the 3-manifold $\mathbb H^3/\Gamma$ contains a totally geodesic surface of genus $g\geq 2$ with enough symmetries such that the  induced monodromy representation $\widehat\rho$ is RL. Note that in the example we give below, the representation $\widehat\rho$ is not Fuchsian. 
\end{remark}

We prove the following (see Theorem \ref{thm9.1}):

\begin{theorem}\label{pentagon}
Let $\Gamma$ be the cocompact  lattice in $\mathbb H^3$ given by the dodecahedron tiling $\{5,3,4\}$ of
the hyperbolic 3-space. Then there exist a holomorphic curve of genus $4$ in the compact 3-manifold SL$(2, \C)/\Gamma.$
\end{theorem}

\begin{remark}
The RSR-representation $\rho$ is obtained from the pentagon tiling of $\mathbb H^2 \subset \mathbb H^3$ which can be extended to the dodecahedron tiling of $\mathbb H^3.$ The genus of the RSR-representation is 4. The holomorphic map obtained from this example has 4 simple branch points and by Riemann-Hurwitz it cannot factor through a lower genus surface.
\end{remark}

The symmetry assumptions in Theorem \ref{maintheorem} ensure that the moduli space 
$\mathcal M_{\R, sym}(\Sigma)$ of compatible equivariant $\SL(2, \R)$-representations, 
which contains $\rho$, is only (real) $1$-dimensional. Moreover, the compatible Riemann 
surface structures on $\Sigma$ are determined by rectangular tori $T^2_\tau \,=\,\C / \Z 
\oplus i\tau \Z$ for $\tau \,\in\, \R_{> 0}$ with one puncture $[0]
\,\in\, T^2_\tau$ given by an 
appropriate quotient of a Hitchin cover of $\Sigma$. The corresponding flat connections 
and their representations can therefore be investigated on $T^2_\tau$ instead of $\Sigma$. 
The details of the setup are explained in Section 
\ref{Setup}. To obtain a holomorphic map into $\SL(2, \C)/\Gamma$ we show that every 
element of $\mathcal M_{\R}^r(T^2_\tau)\cong \mathcal M_{\R,sym}(\Sigma)$ can be realized 
as the monodromy representation of a logarithmic connection $\nabla^H$ on a specific rank 
2 parabolic bundle $H$. Then, this connection over $T^2_\tau$ is shown to lift to a 
holomorphic connection on $\mathcal O^{\oplus2}$ over the Riemann surface $\Sigma $ with 
induced RL-monodromy $\widehat \rho.$

In a  first step we therefore construct a logarithmic connection $\nabla^H(\tau)$ with $\SL(2, \R)$-monodromy for every punctured torus $T^2_\tau\setminus\{ o \}$, compare with Theorem \ref{holoconnection}, with the prescribed parabolic structure $H$. This theorem is of independent interest and it is proven in Section  \ref{holo}. The main observation is that grafting, a procedure generating new real projective structures \footnote{complex projective structures with real monodromy}  from old ones,
 changes the induced spin structure on the surface. In this context it is very important to note that we perform grafting not on flat  projective bundles, but on their lifts
to flat vector bundles which yield the different spin structures.
We recall grafting  for compact Riemann surfaces in Section \ref{grafting} together with its straight forward generalization to the case of a 1-punctured torus. 
Using abelianization on a fixed $T^2_\tau$ and the fact that (every connected component of) the moduli space $\mathcal M_{\R}^r$  is diffeomorphic to $\R_{>0}$, we obtain (via  the intermediate value theorem) that 
there exists  a holomorphic connection $\nabla^H(\tau)$ (on the prescribed parabolic bundle) with real monodromy 
between the
 uniformization connection $\nabla^U$ of $T^2_\tau$ and another oper connection $\nabla^G$ of $T^2_\tau$ obtained from a simple-grafting of the uniformization connection of a different Riemann surface $T^2_{\tau^{uG}}$.

In a  second step,  fixing  a given representation $\rho\in \mathcal M^r_\R$, we start at an appropriate initial configuration of $\nabla^U, \nabla^H$ and $\nabla^G$ and vary the Riemann surface structure $\tau\in \R_{>0}$ of $T^2_\tau$. We show in Lemma \ref{Lemma:fullcomponent} and Lemma \ref{teichho} that both connections $\nabla^U(\tau)$ and $\nabla^G(\tau)$ sweep out the 1-dimensional moduli space of real representations $\mathcal M_{\R}^r$ (which contains $\rho$). Moreover,  $\mathcal M_{\R}^r$ is an ordered space and the ordering between the three connections $\nabla^U$, $\nabla^H$  and $\nabla^G$ is preserved through a continuous deformation.
Since furthermore the Riemann Hilbert mapping is a local diffeomorphism on the 1-punctured torus, see Lemma \ref{localdiffeo}, the dependence of $\nabla^H(\tau)$ in $\tau$ can be chosen to be continuous. Up to technicalities (which are taken care of), due to the fact that  $\tau$ is not necessarily a global coordinate on the submanifold $\mathcal M^r_\R\subset \mathcal M_{1,1}^r,$ the connection $\nabla^H(\tau)$ must also sweep out the moduli space $\mathcal M^r_\R$. In particular,  there exists a value   $\tau_0$ such that the  monodromy representation of $\nabla^H(\tau_0)$ is the prescribed representation  $\rho\in \mathcal M^r_\R$.
By replacing $\nabla^U(\tau)$ and $\nabla^G(\tau)$ with multiple graftings (which differ by a simple-grafting), our arguments show that for every $\rho \in \mathcal M_\R^r$ there exist infinitely many different $\tau_n\in \R_{>0},$ $n\in \N$, with holomorphic connections
$\nabla^H(\tau
_n)$ having  the same monodromy $\rho.$ To explain Remark \ref{remark-infty}, note that the Riemann surface structures of the tori $T^2_{\tau_n}$ and the Riemann surfaces
\eqref{Sigmaa} both 
degenerate as $n\to\infty$ and we obtain infinitely many different  holomorphic curves in the quotient of $\mathrm{SL}(2,\C)$ by the compatible cocompact lattice $\Gamma$.

Theorem \ref{maintheorem} would worth little if we could not
prove the existence of at least one RSR-representation compatible with a cocompact lattice $\Gamma.$ In the last section (Theorem \ref{thm9.1}) we explicitly construct an RSR-representation  and show that it is compatible with the  cocompact lattice of the dodecahedron  $\{5,3,4\}$ tessellation of the hyperbolic 3-space. We expect many more examples by investigating the existence of totally geodesic surfaces inside compact hyperbolic $3$-manifold with RSR-monodromy.

\section{Abelianization on Symmetric Riemann surfaces}\label{Setup}

\subsection{ Logarithmic connections and parabolic bundles}\label{ss:logconn}

Consider a compact connected Riemann surface $\Sigma$. Its
canonical line bundle will be denoted by $K_{\Sigma}$, while
${\mathcal O}_\Sigma$ will denote the sheaf of holomorphic functions on $\Sigma$.
A holomorphic $\rm{SL}(2, \C)$--bundle over $\Sigma$ is a rank 
two holomorphic vector bundle $V\,\longrightarrow \,\Sigma$ such that
the determinant line bundle $\det V\,:=\, \bigwedge^2 V$ is holomorphically
trivial.

Let ${\mathbf D}\,=\, p_1+\ldots +p_n$ be an effective reduced divisor, i.e.,  
the points $p _j \,\in\, \Sigma$ are 
pairwise distinct. Denote by ${\overline\partial}_V$ the Dolbeault operator for a 
holomorphic $\rm{SL}(2, \C)$--bundle $V$ on $\Sigma$; so the kernel of ${\overline\partial}_V$ 
defines the sheaf $\mathcal V$ of holomorphic sections of $V$. A \textit{logarithmic} $\rm{SL}(2, \C)$--connection 
$\nabla\,=\,\overline{\partial}_V+\partial^\nabla$ on $V$ with polar part contained in the 
divisor $\mathbf D$ is a holomorphic differential operator
$$
\partial^\nabla\, :\, \mathcal V\, \longrightarrow\, \mathcal V\otimes K_\Sigma\otimes {\mathcal O}_\Sigma({\mathbf D})
$$
such that
\begin{itemize}
\item the Leibniz rule $\partial^\nabla(fs)\,=\, f\partial^\nabla(s)+ s\otimes 
\partial{f}$ holds for all $s \in \mathcal V$ and $f \in {\mathcal O}_\Sigma$, and

\item the induced holomorphic connection on $\det V\,=\, {\mathcal O}_\Sigma$ 
coincides with the de Rham differential $d$ on ${\mathcal O}_\Sigma$.
\end{itemize}
Note that all logarithmic connections on 
$\Sigma$ are necessarily flat. At every point $p _j$
in the singular divisor ${\mathbf D}$ 
of a logarithmic $\rm{SL}(2, \C)$--connection $\nabla$ on $V$, the associated residue 
 \[\text{Res}_{p_j}(\nabla)\,\in\,\text{End}(V_{p_j})\] is tracefree. Let
$\rho_{j}$ and $-\rho_{j}$ be the eigenvalues of $\text{Res}_{p_j}(\nabla)$; the
logarithmic connection $\nabla$ is called non-resonant if $2\rho_{j}\,\notin\, \mathbb Z$
for all $j\,=\, 1,\,\cdots,\, n$. In the non-resonant case,
the local monodromy of $\nabla$ around $p_j$ is 
conjugate to the diagonal matrix with entries $\exp(\pm 2\pi i\rho_{j})$ 
(see \cite[p.~53, Th\'eor\`eme 1.17]{De}).

A {\it parabolic structure } $\mathcal P$ on a $\rm{SL}(2, \C)$--bundle $V$
over the divisor ${\mathbf D}$ is defined by a collection of complex  lines $L_j\, \subset\, V_{p_j}$
together with parabolic weights $r_j\, \in\, ]0,\, \tfrac{1}{2}[$ for all
$j\,=\, 1,\,\cdots, \,n$. For a parabolic structure $\mathcal P$, the divisor ${\mathbf D}$ is called the parabolic
divisor and $\{L_j\}_{j=1}^n$ are called the quasiparabolic lines.
The parabolic degree of  a holomorphic line subbundle $W\, \subset\, V$  is defined to be
\[\text{par-deg}(W)\,:=\, {\rm deg}(W)+\sum_{j=1}^n r^W_j\, ,\]
where $r^W_j\,= \, r_j$ if $W_{p_j}\,=\,L_j$, and $r^W_j\,=\,
-r_j$ if $W_{p_j}\,\neq\, L_j$.

\begin{definition}[\cite{MS, MY}]
A parabolic structure $\mathcal P$
on the $\rm{SL}(2, \C)$--bundle $V$ is called {\it stable} (respectively, {\it semistable}) if
$\text{par-deg}(W)\, <\, 0$ (respectively, $\text{par-deg}(W)\, \leq\, 0$)
for every holomorphic line subbundle $W\,\subset \, V$. A semistable parabolic bundle that is not stable is called \textit{strictly semistable}.
A parabolic bundle which is not semistable is called \textit{unstable}.
\end{definition}

Any non-resonant logarithmic $\rm{SL}(2, \C)$--connection $\nabla$ on $V$ for which 
all the residues have their eigenvalues in the interval
$(-\tfrac{1}{2},\, \tfrac{1}{2})$ induces a parabolic 
structure $\mathcal P$ on $V$. The parabolic divisor of $\mathcal P$ is
the singular locus $\mathbf D\,=\, p_1 
+\ldots + p_n$ of $\nabla$. The parabolic weight
at $p _j$ is the positive eigenvalue $\rho_j$ of $\text{Res}_{p_j}(\nabla)$ and
the quasiparabolic line at $p_j$ is the eigenline for $\rho_j$.
 
A {\it strongly parabolic Higgs field} on a parabolic
$\rm{SL}(2, \C)$--bundle $(V,\,{\mathcal P})$ is a holomorphic section
$$
\Phi \, \in\, H^0(\Sigma,\, \text{End}(V)\otimes K_{\Sigma}\otimes {\mathcal O}_{\Sigma}({\mathbf D}))
$$
such that $\text{tr}(\Phi)\,=\,0$ and
$$
\Phi(p_j)(V_{p _j})\, \subset\, L _j\otimes (K_{\Sigma}\otimes {\mathcal O}_{\Sigma}({\mathbf 
D}))_{p_j}
$$
for all $j\,=\, 1,\,\cdots,\, n$. These conditions imply that $\Phi(p_j)$ is nilpotent with the quasiparabolic lines 
$L _j\, \subset\, {\rm kernel}(\Phi(p_j))$ for all $j = 1,..., n$.

Two non-resonant logarithmic $\rm{SL}(2, \C)$--connections $\nabla_1$ and $\nabla_2$
on $V$ with polar part contained in ${\mathbf D}\,=\, p_1+\ldots +p_n$ induce the same parabolic structure on $V$ if and only if
$\nabla_1-\nabla_2$ is a strongly parabolic Higgs field for the parabolic
structure given by $\nabla_1$ (or equivalently, for the parabolic
structure given by $\nabla_2$).

A general result of Mehta and Seshadri \cite[p.~226, Theorem 4.1(2)]{MS}, and Biquard 
\cite[p.~246, Th\'eor\`eme 2.5]{Biq} (see also \cite[Theorem 3.2.2]{Pir}) implies that the 
above construction of associating a parabolic bundle to a logarithmic connection
actually produces a bijection between the space 
of isomorphism classes of irreducible flat ${\rm SU}(2)$--connections on $\Sigma \setminus 
{\mathbf D}$ and the stable parabolic $\rm{SL}(2, \C)$--bundles on $(\Sigma,\, {\mathbf D})$. As a 
consequence, every logarithmic connection $\nabla$ on $V$ giving rise to a stable parabolic 
$\rm{SL}(2, \C)$--structure $\mathcal P$ admits a unique strongly parabolic Higgs field
$\Phi$ on $(V,\, {\mathcal P})$ such that the monodromy representation of  $\nabla+\Phi$ is unitary.

\subsection{Flat $\rm{SL}(2, \C)$--connections on the $4$-punctured sphere} \label{subs-4 punctured}

Let ${\mathcal S}_4$ denote the Riemann sphere $\CP^1$ with four unordered  marked points
$$
{\mathcal S}_4 \, :=\, (\CP^1,\, \{p_1,\,\cdots ,\, p_4\})\, ,
$$
with $p _j$ as in \eqref{rp}
and  recall that 
\begin{equation}\label{s4}
S_4\, =\, \CP^1\setminus\{p_1,\,\cdots ,\, p_4\}\,
\end{equation}
is the underlying topological four-punctured sphere. Fix a base point $s_0\,\in\, S_4$. For every $j=1,\,...,\,4$, consider a simply closed,  
oriented, and $s_0$-based loop $\gamma_{p _j}$ going around a single puncture $p _j$.
The fundamental group $\pi_1(S_4, \,s_0)$ 
is generated by these curves $\gamma_{p _j}$ with $j=1,\,...,\,4$ and they satisfy the relation 
$\gamma_{p_4}\gamma_{p_3}\gamma_{p_2}\gamma_{p_1}\,=\,1$.

\noindent
\textbf{Convention.}\, For convenience, the composition of loops generating the fundamental
group operation is considered to be from right to left, i.e., $\gamma_2\gamma_1$ denotes the loop obtained by first
performing the loop $\gamma_1$ and then $\gamma_2$.

Every $\SL(2,\C)$-representation of $\pi_1(S_4, \,s_0)$ is determined by the images 
$M _j\,\in\,\SL(2,\C)$ of the generators $\gamma_{p _j} \,\in\, \pi_1(S_4, \,s_0)$, for 
$j=1,\,...,\,4$ and we have
\[M_4M_3M_2M_1\,=\,I.\]

We restrict to the symmetric case where 
\[\Tr(M _j)\,=\,2\cos(2\pi \wt{r}) ,\;\quad\forall\,\,\, \;j=1,\,...,\,4, \]
with $\wt{r}\in (\frac{1}{4},\,\frac{1}{2}).$ We denote by $\mathcal M_{0,4}^{\wt r}$ the space of equivalence classes of 
$\SL(2,\C)$-representation of $\pi_1(S_4, \,s_0)$ with local monodromy satisfying the above 
condition at all punctures. This $\mathcal M_{0,4}^{\wt r}$ is identified with the space of 
flat $\SL(2,\C)$-connections on the four-punctured sphere such that all four local monodromies 
are in the same conjugacy class given by
 \begin{equation}
\dvector{ e^{-2\pi i \wt r} &0 \\ 0& e^{2\pi i\wt r}}\,\in\,
{\rm SL}(2,{\mathbb C})\, .
\end{equation}

For $\rho\in \mathcal M_{0,4}^{\wt r}$, we denote by 
\[\wt x\,=\,\Tr(M_2M_1),\,\ \wt y\,=\,\Tr(M_3M_2),\,\ \wt z=\Tr(M_3M_1)\]
its {\em trace coordinates}.
They satisfy 
\begin{equation}\label{Fricke4}
\wt x^2+
\wt y^2 +\wt z^2 +\wt x\wt y\wt z-2\mu^2 (\wt x+\wt y+\wt z)+4(\mu^2 -1)+\mu^4 \,=\,0.
\end{equation}
The corresponding affine variety is called a  (relative) {\em character variety}. The following result of characterizing a representation by its image in the corresponding 
relative character variety is well-known and dates back to Fricke and Klein, see \cite{Gold88,BeG}. 
\begin{lemma}\label{lemma_uniqueness}
Let $(\wt x, \wt y, \wt z) \in \C^3$ satisfying equation \eqref{Fricke4}
and
$(\widetilde{x}-2)(\widetilde{y}-2)(\widetilde{z}-2)\,\neq\, 0$. Then, there exist a unique $\rho\in \mathcal M_{0,4}^{\wt r}$ such that $(\wt x, \wt y, \wt z)$ are the trace coordinates of $\rho$.
\end{lemma}

Moreover, a totally reducible representation is conjugate to a $\mathrm{SU}( 
2)$-representation if and only if $\wt x,\,\wt y,\,\wt z\,\in\,[-2,\,2]$, while it is 
conjugate to an $\SL(2,\R)$-representation if $\wt{x},\,\wt{y},\,\wt{z}\,\in\,\R$ are real and
at least one of them lying in $\R \setminus [-2,\,2]$.

\begin{remark}\label{character4sym}
For the parabolic weight $\wt{r}\,\in\,(0,\, \frac{1}{4})$, there is a natural biholomorphic map between the
character varieties for $\wt{r}$ and $\tfrac{1}{2}-\wt{r}$. In fact, this
biholomorphism is induced by  $M_k\,\longmapsto\, -M_k$,
which gives the identity map in terms of the respective $(\wt x,\,\wt y,\,\wt z)$-trace coordinates.
Note that 
\[2\text{cos}(2\pi \wt{r})\,=\,-2\text{cos}(2\pi (\tfrac{1}{2}-\wt{r})),\]
and therefore also equation \eqref{Fricke4} does not change. 
\end{remark}

\subsection{Flat $\rm{SL}(2, \C)$--connections on the 1-punctured torus}\label{sect:SL2onT}

For $\tau\,\in\, \R_{>0}$ let \begin{equation}\label{et}
T_{\tau}^2\,:= \,\C / \Gamma, \quad \text{with} \quad \Gamma\,=\, {\mathbb Z}+\tau i{\mathbb Z}\,\subset\,
\mathbb C
\end{equation}
 be a rectangular torus. Moreover, let $o\,=\,[0]\,\in\, T_{\tau}^2$ and  
$p_0\, :=\, \frac{1+ \tau i}{4}\,\in\, T_{\tau}^2$ and consider $\pi_1(T_{\tau}^2\setminus\{o\},\, p_0)$  the fundamental group of
the one-punctured torus $T_{\tau}^2\setminus \{o\}$ with basepoint $p_0$.
This is a free group with two generators 
$\gamma_x,\,\gamma_y\,\in\,\pi_1(T_{\tau}^2\setminus\{o\}, \, p_0),$ 
where
\begin{equation}\label{gamma_x1}
\gamma_x\,\colon\, [0,\, 1]\,\longrightarrow\, T_{\tau}^2\setminus\{o\};\ \, s\,\longmapsto\,
s+
\frac{1+\tau i}{4}
\end{equation}
and
\begin{equation}\label{gamma_x}
\gamma_y\,\colon\, [0,\,1]\,\longrightarrow\, T_{\tau}^2\setminus\{o\};\ \, s\,\longmapsto\,
\tau is+\frac{1+ \tau i}{4}.
\end{equation}
The commutator $\gamma_y^{-1}\gamma_x^{-1}\gamma_y\gamma_x \,\in\,
\pi_1( T_{\tau}^2\setminus\{o\}, \, p_0)$ corresponds to a simple loop 
going around the marked point $o$ anti-clockwise.

For $r \in(0,\,\tfrac{1}{2})$ let $\mathcal M_{1,1}^{r}$ be the moduli space of
flat $\mathrm{SL}(2,\C)$-connections on the 1-punctured torus $T^2_\tau \setminus \{o\}$
with local monodromy around the marked point $o$ lying in the conjugacy class of the matrix 
\begin{equation}\label{locmon}
\dvector{ e^{-2\pi i r} &0 \\ 0& e^{2\pi i r}}\,\in\,
{\rm SL}(2,{\mathbb C})\, .
\end{equation}
As for the 4-punctured sphere, the conjugacy class is determined by the value of its trace $2\cos(2 
\pi r)$, see \cite{Gold}. For an element in $\mathcal M_{1,1}^{r}$ let $X,\,Y$ be the  monodromies along the curves
$$\gamma_x, \,\gamma_y \,\in\, \pi_1( T^2_\tau \setminus \{0 \}, \, p_0)$$
(defined in \eqref{gamma_x1} and \eqref{gamma_x}), and 
let
\[x\,=\,\Tr(X),\quad y\,=\,\Tr(Y),\quad z\,=\,\Tr(YX)\]
be the corresponding trace coordinates satisfying  the equation
\begin{equation}\label{character-equation}
x^2+y^2+z^2-xyz-2-2\cos(2\pi r)\,=\,0\,.
\end{equation}
The corresponding affine variety is also called  the (relative) character variety of the 1-punctured torus.
By a result of Fricke, the moduli space $\mathcal M_{1,1}^{r}$ is diffeomorphic
to the character  variety defined in (\ref{character-equation}) by  associating to a  monodromy representation the traces
$x\,=\,\Tr(X),\,y\,=\, \Tr(Y)$ and $z\,=\,\Tr(YX)$  (see, \cite[Section 2.1]{Gold}).

\begin{remark}\label{2.2}
For given $x$ and $y$, the equation \eqref{character-equation} is quadratic in $z$, and 
hence there are two (possibly equal) solutions of \eqref{character-equation} in $z$ which 
we refer to as $z_1$ and $z_2$. If $r$, $x$ and $y$ are all real, then $z_1$ and $z_2$ are 
complex conjugate to each other.
\end{remark}
 
\begin{theorem}[\cite{Gold}]\label{goldman}
For $r\,\in\,(0,\,\tfrac{1}{2})$ fixed, the space of all real points of  the character
variety defined by \eqref{character-equation} has
5 connected components: one compact component characterized by
the condition $x,\,y,\,z\,\in\,[-2,\,2]$ and  four non-compact 
components (which are all diffeomorphic to each other). The compact component consists of $\mathrm{SU}(2)$-representations,
while the non-compact components consist of $\mathrm{SL}(2,\R)$-representations.
\end{theorem}

\begin{remark}
The non-compact components are interchanged by tensoring with a flat 
$\Z_2$-bundle. These are called sign-change automorphisms by Goldman
\cite{Gold}.
\end{remark}

\subsection*{A map between the character varieties}

By  \cite[Theorem 4.9]{BDHH} (see also \cite{HeHe}), there exists for every $r \,\in
\,(0,\,\tfrac{1}{2})$ a degree 4  birational map between the moduli space
$\mathcal M^{r}_{1,1}$ of flat $\mathrm{SL}(2,\C)$-connections on the one-punctured
torus $T^2_\tau$ (as defined in \eqref{et}) and the moduli space $\mathcal M_{0,4}^{\wt{r}}$
of flat $\mathrm{SL}(2,\C)$-connections on the four-punctured sphere $S_4$ (as defined
in \eqref{S4}) with \[\wt{r} \,=\, \tfrac{1+2r }{4}\,\in\, (\tfrac{1}{4},\,\tfrac{1}{2}).\]
On the level of character 
varieties  this  map $ \mathcal M^{r}_{1,1} \longrightarrow \mathcal M_{0,4}^{\wt{r}}  $   is given by 
\begin{equation}\label{abelintraces}(x,\,y,\,z)\,\longmapsto\, (\wt x,\,\wt y,\,\wt z)\,=\,(2-x^2,\,2-y^2,\,2-z^2).\end{equation}

The construction of the above map in \cite[Theorem 4.9]{BDHH}
uses the rectangular torus $T^2_\tau$, with $\tau\in \R_{>0},$  being a double cover of
$\CP^1$ branched over the four points $p_1,\, \cdots ,\, p_4$ (defined in \eqref{rp}), i.e.,  $T^2_\tau$ is given by
 \[y^2\,=\,\frac{(z-p_1)(z-p_2)}{(z-p_3)(z-p_4)}.\]
 
 Define \begin{equation}\label{4ps}
{\mathcal S}_4 (\tau) \, :=\, (\CP^1,\, \{p_1,\,\cdots ,\, p_4\})\, ,
\end{equation}
with $p _j$ as in \eqref{rp}, chosen to define a rectangular torus $T^2_\tau$, with $\tau
\,\in \,\R_{>0}$. Since $T_{\tau}^2$ is rectangular, the reflection along one edge 
\begin{equation}\label{eet}
\eta\,\colon\, T_{\tau}^2\,\longrightarrow\, T_{\tau}^2,\quad [w]\,\longmapsto\, [\overline{w}],
\end{equation}
where $w$ is the global coordinate on $\C,$ defines a real involution on $T_{\tau}^2$.

\begin{remark}
Since the real involution $\eta$ considered here is different than in \cite{BDHH},
we also use slightly different coverings of the 4-punctured sphere, see Lemma \ref{trivialoncovering}. 
Nevertheless, the main results in \cite{BDHH} to obtain Fuchsian representations on the holomorphically trivial bundle remain true by analogous arguments.
\end{remark} 

\subsection{Abelianization}\label{sect:abelianization}
Every element in $\mathcal 
M_{1,1}^r$ can be represented (meaning it lies in the same smooth  gauge class)
by a logarithmic flat connection with a simple pole at $o$.
Abelianization yields particularly well-behaved coordinates $a,\,\chi \,\in\, \C$
on ${\mathcal M}_{1,1}^r$ as follows, see also
\cite[Section 4]{BDH}, or \cite{HeHe}. For  $L$ being the $C^\infty$-trivial bundle $T_{\tau}^2\times {\mathbb C}\, \longrightarrow\, 
T_{\tau}^2$ the generic logarithmic connection in $\mathcal 
M_{1,1}^r$ is given by 
\begin{equation}\label{abel-connection}
\nabla\,=\,\nabla^{a,\chi,r}\,=\,\dvector{\nabla^L &\gamma^-_\chi\\ \gamma^+_\chi &
\nabla^{L^*} }\,,
\end{equation}
where 
\begin{equation}\label{nablaL}
\nabla^L\,=\,d+adw+\chi d\overline{w}
\end{equation}
 is the flat connection on $L$ for constant $a,\,\chi \in \C$ and $w$ being the global holomorphic coordinate on 
$T_{\tau}^2$ (see \cite[Section 4]{BDH}). Moreover, $\nabla^{L^*}$ is the dual connection 
of $\nabla^L$, and the induced holomorphic structure  (by \eqref{abel-connection})
on $L$ is given by the Dolbeault operator 
$\overline{\partial}^0 + \chi d\overline{w}$, where $\overline{\partial}^0\,=\,d''$ is the 
$(0,1)$-part of the de Rham differential operator $d$.
In this generic case,  characterized by $\chi$  not being a half-lattice point of Jac$(T^2_\tau)$, i.e., $L^2\neq \mathcal O_{T_{\tau}^2},$
 $\gamma^+_\chi$ and $\gamma^-_\chi$ are meromorphic sections with 
respect to the holomorphic structures given by the Dolbeault operators 
$$\overline{\partial}^0 - 2 \chi d\overline{w}\quad  \text{ and } \quad \overline{\partial}^0 + 2 \chi 
d\overline{w},$$ respectively, with simple poles at $o\,\in\, T^2_\tau$ and  residues determined 
by the eigenvalue of the local monodromy $r\in (0, \frac{1}{2})$. 

In the non-generic case, the underlying rank two holomorphic bundle is a non-trivial extension of a spin bundle $S$ by itself.
With respect to the $C^\infty$-splitting $S\oplus S^*$ 
the Dolbeaut operator 
 is given by
\begin{equation}\label{abel-connection-spindbar}
\overline\partial^\nabla\,=\,\dvector{\overline\partial^0+\chi d\overline{w} &d\overline{w}\\ 0&
\overline\partial^0-\chi d\overline w }
\end{equation}
for a half lattice point $\chi \in$ Jac$(T^2_\tau)$, while the $\partial$-part 
\begin{equation}\label{abel-connection-spinpartial}
\partial^\nabla\,=\,\dvector{\partial^S & b dw\\ c dw&
(\partial^S)^* }\,
\end{equation}
is singular at $o$, i.e., $\partial^S $ and its dual $(\partial^S )^*$ are line bundle connections singular at $o$ and $b$ is a function singular at $o,$ 
but $c\in\C^*$ is a non-zero constant.
Thus, in the non-generic case, the connection  takes the form of an (orbifold) oper, compare with \eqref{eq:oper} below.

\begin{remark}
For given $\nabla$ the holormorphic structure on $L$, denoted by  $\chi(\nabla)\in$ Jac$(T^2_\tau)$ by abuse of notation, is only well defined up to taking the dual. \end{remark}

\begin{remark}
The parabolic weight at the puncture $o$ induced by the logarithmic connection $\nabla$   is given by $r.$ 
For $\nabla=\nabla^{a,\chi,r}$, where $\chi$ is not a half-lattice point,
the parabolic 
line $l_o$ is uniquely determined, up to a holomorphic automorphism of $L\oplus L^*$, by the condition that it is neither the line 
$L_o$, nor the line $L^*_o$ (i.e., the off-diagonal is non-zero). 
In the non-generic case, the parabolic line is either given by $l_o=S_0$ and the parabolic bundle is unstable, or the parabolic line is not contained in the unique holomorphic line subbundle of degree 0, and the parabolic bundle is stable.
\end{remark}

Note that $o$ is contained in the fix point set of the reflexion  $\eta$ in \eqref{eet}.
Since $r\in(0,\, \frac{1}{2})$ is  real, $\eta$ 
 induces a real involution of the corresponding de Rham moduli space
\[\widehat{\eta}\,\colon\, \mathcal M_{1,1}^{r}\,\longrightarrow\,
\mathcal M_{1,1}^{r},\ \quad [\nabla]\,\longmapsto\, [\eta^*\overline\nabla]\]
 and we have

\begin{lemma}\label{tausymcon}
On the rectangular torus $T_{\tau}^2$ the gauge class of a connection 
$\nabla\,=\,\nabla^{a,\chi, r}$, with $\chi\notin\tfrac{1}{2}\Gamma$, is
fixed by the involution $\eta$ if and only if one of the following four conditions holds:
\begin{equation}\label{realLL}
\begin{split}
&\chi\in \R,\;\text{ and }\,\; a\,\in\, \R, \\
\text{or }\quad &\chi+k \frac{\pi i}{2}\,\in\, \R\,\;\text{ and }\; a-k \frac{\pi
i}{2}\,\in\, \R \;\;\text{for some }\;k\,\in\,\Z ,\\
\text{or }\quad&\chi\,\in\,  i\R\,\; \text{ and }\; a \,\in\, i\R, \\
\text{ or }\quad  &\chi+k\frac{\pi  }{2\tau}\,\in\, i\R\,\;\text{ and }\; a- k
\frac{\pi }{2\tau}\,\in\, i\R\;\;\text{for some }\;k\,\in\,\Z.
\end{split}\end{equation}
\end{lemma}

\begin{proof}
We have
\[\eta^* dw\,=\, d\overline{w}\quad \text{and}\quad \eta^* d\overline{w}\,=\,d w.\]
Hence, for $\chi\,\in\, \R$ and $a\,\in\, \R$, the connection $\nabla^L$ in \eqref{nablaL}
satisfies the condition
\[\eta^*\overline{\nabla^L}\,=\, \nabla^L .\]
By \cite{HeHe,BDH} the sections
$\gamma^\pm_\chi$ in \eqref{abel-connection} are unique up to scaling.
Moreover, the quadratic residue of
\[\gamma^+_\chi\gamma^-_\chi (dw)^2\]
is $r^2$, and hence this residue determines the conjugacy class of the monodromy
of $\nabla$ around the singular point $o\,\in\, T_{\tau}^2$.
Thus, we obtain constants $c^+,\,c^-$ with
\[\eta^*\overline{\gamma^\pm_\chi dw}\,=\,c^\pm\gamma^\pm_\chi\]
and consequently the two connections $\nabla$ and $\eta^*\overline\nabla$ are gauge
equivalent. The argument for the other $3$ cases in \eqref{realLL} works analogously.

Conversely, if the pull-back $\eta^*\overline\nabla$ is gauge equivalent to $\nabla$ then (by
\cite{HeHe,BDHH}) $\eta^*\overline{\nabla}^L$ is gauge equivalent to either $\nabla^L$ or its
dual. This yields that $a$ and $\chi$ must satisfy one of the conditions in \eqref{realLL}
\end{proof}

For latter purposes, we denote the space of $\eta$-invariant representations by
\[\mathcal M_{1,1}^{r,\eta}\,:=\,\{[\nabla]\,\in\, \mathcal M_{1,1}^{r}
\,\,\mid\,\,\, [\eta^*\overline\nabla]\,=\,[\nabla]\}.\]

\subsection{The hidden symmetries of RSR-representations}\label{sec_symmetries}

\begin{proposition}\label{prop1}
For $\wt{r}
\,\in\, (\tfrac{1}{4},\, \tfrac{1}{2})$ consider  $\wt\rho \in \mathcal{M}_{0,4}^{\wt{r}}$ and let $(\wt{x},\, \wt{y},\,\wt{z})$ be  the corresponding trace coordinates satisfying \eqref{Fricke4} and $(\wt{x}-2)(\wt{y}-2)(\wt{z}-2)
\,\neq\, 0$ (as in Lemma \ref{lemma_uniqueness}). Then there exist an  
element $(x,\,y,\,z)$ in $\mathcal{M}_{1,1}^{r}$,  unique up to signs, satisfying \eqref{character-equation}
with $r\,=\,2 \wt{r} -\frac{1}{2}$ such that
\begin{equation}\label{eqn32}
\wt x\,=\,2-x^2,\quad \wt y\,=\,2-y^2,\quad \wt z\,=\,2-z^2.
\end{equation}
\end{proposition}

\begin{proof}
As in the proof of Theorem 4.9 in
\cite{BDHH}, the Klein-Fricke equation \eqref{Fricke4} factors into the product of the Klein-Fricke
equation \eqref{character-equation} for $(x,\,y,\,z)$ and the Klein-Fricke
equation \eqref{character-equation} for $(x,\,y,\, -z) $ when applying \eqref{eqn32}. Hence,
for given $(\wt x,\, \wt y,\, \wt z)$ either the corresponding $(x,\,y,\,z)$ or $(x,\,y,\, -z)$
solves equation \eqref{character-equation}.
\end{proof}

\begin{remark}
Proposition \ref{prop1} shows the existence of additional symmetries of representations
$\wt\rho \in \mathcal{M}_{0,4}^{\wt{r}}$ on the 4-punctured sphere, see
equation (4.19) in \cite{BDHH} or \cite[$\S~6$]{Gold97}.
\end{remark}

\begin{corollary}
For $\wt{r}\,\in
\,(\tfrac{1}{4}, \,\tfrac{1}{2})$ let $\rho \in \mathcal{M}_{0,4}^{\wt{r}}$ be a RSR-representation. Then $\rho$  corresponds via abelianization to
a real representation in $\mathcal{M}_{1,1}^{r},$ with $r\,=\,2 \wt{r} -\frac{1}{2}$. By abuse of notation we will denote the induced representation in $\mathcal{M}_{1,1}^{r},$ by $\rho$ as well.
\end{corollary}
\begin{proof}

By Definition \ref{RSR-def}, the image of $\rho$ in the character variety \eqref{Fricke4}
is a point $(\wt x,\, \wt y,\, \wt z)$ such that 
$\wt x\,<\,-2$,\, $\wt y \,<\,-2$ and $\wt z\,<\,-2$. Hence the corresponding solution
$(x,\,y,\,z)$ of \eqref{eqn32} is a real  point in the character variety of the the 1-punctured torus defined by \eqref{character-equation}. By \cite{Gold} , see Theorem \ref{goldman}, this solution $(x,\,y,\,z)$ corresponds  to  a real element
in $\mathcal{M}_{1,1}^{r}$, where $r\,=\,2 \wt{r} -\frac{1}{2}$.
 \end{proof}

\subsection{Strictly semi-stable parabolic bundles on the 4-punctured sphere}
On ${\mathcal S}_4$, the Riemann sphere $\CP^1$ with four unordered  marked points,
 fix the parabolic weight $\wt{r}
\,\in\,(\tfrac{1}{4},\, \tfrac{1}{2})$ to be the same at each puncture. Then,
up to isomorphism, there are 
exactly 3 strictly semi-stable parabolic rank 2 bundles with trivial determinant and given parabolic weights. These 3 parabolic bundles are defined on the 
holomorphic rank 2 bundle $\mathcal O\oplus\mathcal O\,\longrightarrow\,\C P^1$, and determined by the choice of signs $\sigma_2,\,\sigma_3,\,\sigma_4\,\in\, \{\pm1\}$ with 
\[1+\sigma_2+\sigma_3+\sigma_4\,=\,0,\]
 induced by the reducible Fuchsian systems
\begin{equation}\label{reducibleFuchs}
D\,=\,d+\begin{pmatrix} \rho&0 \\ 0& -\rho\end{pmatrix}\left( 1 \frac{dz}{z-p_1}+ \sigma_2 \frac{dz}{z-p_2}+ \sigma_3 \frac{dz}{z-p_3}+ \sigma_4 \frac{dz}{z-p_4} \right).\end{equation}
Moreover, they admit strongly parabolic  and offdiagonal Higgs fields $\Phi$ with non-zero
determinant, e.g., for $\sigma_2 = 1, \sigma_3=\sigma_4 =-1$ we have
\begin{equation}\label{specialHiggsfield}
\Phi\,=\,\begin{pmatrix} 0& 1\\0&0\end{pmatrix}\left (\frac{dz}{z-p_1}-\frac{dz}{z-p_2} \right)
+\begin{pmatrix} 0& 0\\1&0\end{pmatrix} \left(\frac{dz}{z-p_3}-\frac{dz}{z-p_4} \right).
\end{equation}

\begin{lemma}\label{Lemma:4-spin}
In the abelianization coordinates, the holomorphic structures, i.e., the $\chi$ 
coordinates, corresponding to the 3 strictly semistable parabolic bundles with 
parabolic weight $\wt{r}$ at every puncture are given by the $4\times$3=12 non-trivial 
4-spin bundles on the torus $T^2_\tau$, i.e., by the  line bundles $L\in
{\rm Jac}(T^2_\tau)$ with $L^{\otimes 4}\,=\, {\mathcal O}_{T^2_\tau}$ and $L^{\otimes 2}\,\neq\, {\mathcal O}_{T^2_\tau}$.
\end{lemma}

\begin{proof}
The case $\sigma_3=1, \sigma_2=\sigma_4 = -1$ is already considered in \cite{HHSch}, and \cite{BDHH}.
We only give the proof for $\sigma_2\,=\,1$ and $\sigma_3\,=\,\sigma_4\,=\,-1$,
as the case $\sigma_4\,=\,1$ and $\sigma_2\,=\,\sigma_3\,=\,-1$ work analogously.  Let $$\pi
\,\colon\, \Sigma_2\,\longrightarrow\, \mathcal S_4$$ be a double covering of the sphere $\C P^1$ branched over the four marked  points
$p_1, ..., p_4$, defined by the equation $$\widetilde{y}^2\,=\,\prod_{i=1}^4 (z-p _j).$$ 

Let $w _j := \pi^{-1}(p _j)$ for $i= 1, ..., 4$ and consider the reducible Fuchsian system $D$ \eqref{reducibleFuchs} on $\mathcal S_4.$ We will show that the line bundle $L$ which determines via abelianization \eqref{abel-connection} the gauge class of 
the connection  $[\pi^*D]$ is given by $L \,=\, L(w_1-w_2)\,\longrightarrow\, \Sigma_2$. That $L$ corresponds to 
a half-lattice point of the Jacobian translates to the condition $L = L^*.$

\begin{remark}
In fact $\Sigma_2\,\cong\,\C/(2\Z+2\tau i \Z)$ is a 4-fold covering of $T^2_\tau$.
 The holomorphic structure on the spin bundle $L(w_1-w_2)\,\longrightarrow\, \Sigma_2$  is given by
\[\overline{\partial}^0-\frac{\pi }{4\tau} d\overline{w} \,\equiv\,\overline{\partial}^0+\frac{\pi }{4\tau} d\overline{w},\]
see e.g. \cite[Section 3]{HeHe}.
In particular, for  the choice of signs $\sigma_2\,=\,1$ and $\sigma_3\,=\,\sigma_4\,=\,-1$
the $\chi$-coordinate of $[\pi^*D] \,\in\, \mathcal M_{1,1}^{r}$  is a real half lattice point of Jac$(\Sigma_2)$. Since $\Sigma_2$ is a 4-fold covering of $T^2_\tau$, the  $1/4$-lattice points of Jac$(T^2_\tau)$ pull back to half lattice points of Jac$(\Sigma_2)$, i.e., to non-trivial spin bundles of $\Sigma_2$.
\end{remark}

\subsubsection*{Proof of Lemma \ref{Lemma:4-spin} continued} The Higgs field $\Phi$ as defined in \eqref{specialHiggsfield} has eigenvalues $\pm c \frac{d z}{\widetilde y}$ for some $c\in \C^*$ and a direct computation shows that its eigenline bundles $E^\pm\to \Sigma_2$  are given by
holomorphic inclusion maps
\[\begin{pmatrix} s_\pm\\ t_\pm\end{pmatrix} \colon E^\pm\to \mathcal O\oplus \mathcal O\]
with divisors
\[(s_+)=(s_-)\,=\,w_3+w_4\quad  \text{and} \quad (t_+)=(t_-)\,=\,w_1+w_2\]
(moreover,  $s_+=-s_-$ and $t_+=t_-$ up to scaling), see  \cite[Theorem 2]{HeHe}.  Therefore, 
$$E^\pm \,=\,  L(-w_1-w_2)\,=\, L(-w_3-w_4).$$
Hence,
the corresponding holomorphic line bundle $L $ of degree 0 obtained after tensoring with  $L(2w_1)$ satisfies $L=L^*$ and
is given by $$L \,=\, L(w_1-w_2)=L(2w_1-w_3-w_4).$$
\end{proof}

The Riemann surface $\Sigma$ considered in this paper is obtained by a covering of the 4-punctured sphere $\mathcal S_4$ with the number of sheets depending on the parabolic weight $\wt r.$  More specifically,  
let $\wt{r}=\frac{l}{k}  \in (\frac{1}{2}, \frac{1}{4})$ with coprime integers $l,k\in\N$.
Fix  $\sigma_2\,=\,1$ and $\sigma_3\,=\,\sigma_4\,=\,-1$.
The Riemann surface $\Sigma= \Sigma_g$ given by a $(g+1)$-fold covering $\pi_g: \Sigma_g \rightarrow \mathcal S_4$ defined by the equation 

\begin{equation}\label{eqRS}\Sigma_g \colon y^{g+1}\,=\,\frac{(z-p_1)(z-p_2)}{(z-p_3)(z-p_4)},\end{equation}
where $g=k-1$ if $k$ is odd and $g=k/2-1$ for $k$ even.
Using this covering the singularities of the connections on $\mathcal S_4$ become apparent on $\Sigma_g.$ In other words, there exist a singular gauge under which the connections become smooth connections on $\Sigma_g$. Likewise 
 \begin{equation}\label{eqnRS2}
 \wt\Sigma_g\colon y^{g+1}=\frac{(z-p_1)(z-p_4)}{(z-p_2)(z-p_3)}
 \end{equation}
 is the compact surface with respect to the sign choice $\sigma_4=1$ and $\sigma_2=\sigma_3=-1$.
 The trivial holomorphic structure on $\Sigma_g$ (and analogously for $\wt \Sigma_g$) can be easily identified according to the following Lemma. 

\begin{lemma}\label{trivialoncovering}
Let $\Sigma_g$ \eqref{eqRS}  and $D$ \eqref{reducibleFuchs} be defined as above (for the same choices of $\sigma_2,\sigma_3,\sigma_4$).
Then there exist a unique $\Z_2$-connection $\nabla^{\Z_2}$  on $\Sigma_g$ with
local  monodromies $(-1)^{k-1}$ around each preimage of the marked points $p_1,\dots,p_4$ such that the flat $\mathrm{SL}(2,\C)$-connection
$\nabla^{\Z_2}\otimes(\pi_g)^*D$  has trivial monodromy on $\Sigma_g,$ i.e. it is gauge equivalent to the trivial (smooth) connection
on the compact Riemann surface $\Sigma_g$. If $k$ is odd, $\nabla^{\Z_2}$ is trivial.
\end{lemma}

\begin{proof}
The proof is completely analogous to the proof of Proposition 3.1 in \cite{BDHH} (see also \cite[Theorem 3.2(5)]{HHSch}) by adapting  the covering  \eqref{Sigmaa} and the reducible Fuchsian systems  \eqref{reducibleFuchs} considered to the different sign choice  $\sigma_2=1$ and $\sigma_3=\sigma_4=-1$ in this paper. The same holds for the choice  $\sigma_4=1$ and $\sigma_2=\sigma_3=-1$. 
\end{proof}

\begin{lemma}\label{tausymcon2}
Let $\nabla=\nabla^{a,\chi, r }$ be a connection on $T_{\tau}^2$ with $[\eta^*\overline{\nabla}]=[\nabla].$ Then
$$x=\Tr(X) \in \R,\quad\text{and} \quad y=\Tr(Y)\,\in\R \quad \text{and} \quad z_1=\Tr(YX) \,=\,
\overline{z}_2=\overline{\Tr(Y^{-1}X) }.$$
In particular, the representation is real if and only if $z_1=z_2 \in \R$.
\end{lemma}
\begin{proof}
The first part is completely analogous  to the proof of Lemma \cite[Lemma 4.6]{BDHH}\footnote{The reader should be aware of the different real involutions $\eta$ used in \cite{BDHH} and in this paper.}.
 The second part is a direct consequence of Theorem \ref{goldman}.
\end{proof}

\begin{remark}\label{41d}
The Lemma implies that an $\eta$-invariant representation is real if and only if the discriminant of \eqref{character-equation}, as quadratic equation in $z,$  is zero.  To be more explicit this gives the extra equation
\begin{equation}\label{extraaffineeq}x^2y^2-4x^2-4y^2+8(1+\text{cos}(2 \pi r))=0.\end{equation}

The main advantage of considering this very symmetric case is that the space of real and  $\eta$-invariant representations becomes real 1-dimensional, see Figure \ref{gammaz}.
The four different non-compact  real components of the character variety (see Theorem \ref{goldman}) correspond to the four different
spin bundles over the torus. The trace coordinates of the  four non-compact components differ only by signs, and these four components are mapped into the same component of real representations of the 4-punctured sphere via abelianization \eqref{abelintraces}.

\begin{figure}[ht]\centering
\includegraphics[width=0.45\textwidth]{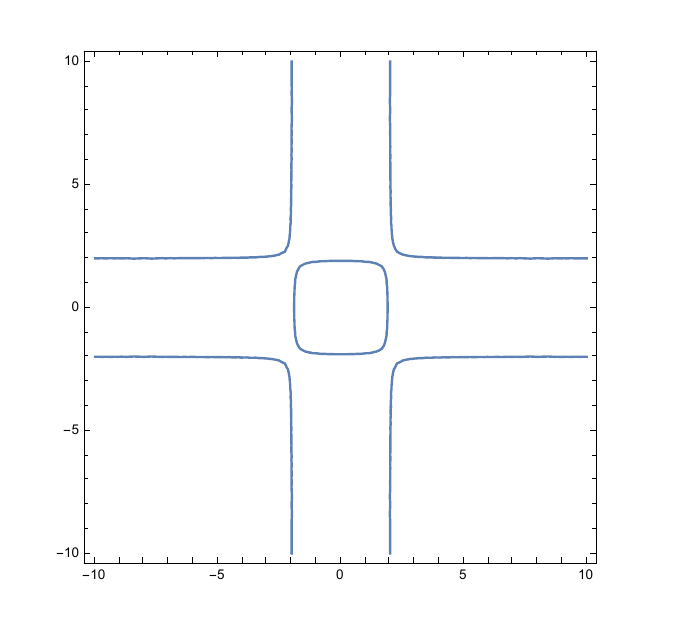}
\caption{The space of real representations invariant under $\eta$, for $r=\tfrac{1}{10}$ ($\wt r = \tfrac{3}{10}$), which is the parabolic weight of our dodecahedron example.}
\label{gammaz}
\end{figure}
\end{remark}

\subsection{Hitchin section}\label{sec:hitchinsec}
Consider a spin structure  
\[S\longrightarrow T_{\tau}^2\]
on the rectangular torus $T_{\tau}^2,$ i.e., $S^{\otimes 2} \cong K_{T^2_\tau}$, and the strictly stable and strongly parabolic Higgs bundle $(E, \wt r, l_0, \Phi)$ given by the data
\[E=S\oplus S;\quad  r \in (0, \tfrac{1}{2});\quad  l_o=\C S_o\oplus 0;\quad \Phi=\begin{pmatrix} 0 & 0\\dw &0\end{pmatrix}.\]
Then Hitchin-Kobayashi correspondence on non-compact curves \cite{Si1}
yields a compatible flat connection $\nabla$ satisfying   \[[\nabla]\in  \mathcal M_{1,1}^{r}\]
with $\mathrm{SL}(2,\R)$-monodromy. The underlying holomorphic bundle is hereby the non-trivial extension of $S$ by itself and the parabolic line $l_0$ is contained in the unique holomorphic line
subbundle $S$ \footnote{The later follows from the fact that the harmonic metric solving the self-duality equations must be diagonal by uniqueness.}.

For every $q\in\C$  a strongly parabolic Higgs field 
\[\Phi^q=\begin{pmatrix} 0 & q dw\\dw &0\end{pmatrix}\]
on the parabolic bundle $(E, r,  l_o)$ corresponds to a compatible flat connection $\nabla^q$ with $[\nabla^q]\in \mathcal M_{1,1}^{r}$
 having $\mathrm{SL}(2,\R)$-monodromy (see Simpson \cite{Si1}) -- this is a particular instance of the so-called {\em parabolic Hitchin-Kobayashi correspondence}.

\begin{lemma} \label{lemma:real Hitchin}
For fixed parabolic weight $r \in (0, \tfrac{1}{2})$, the element   $[\nabla^q]\in \mathcal M_{1,1}^{r}$ is $\eta$-invariant if and only of $q \in \R.$  
\end{lemma}
 
\begin{proof} The parabolic Higgs pairs $(E, r,  l_o, \Phi^q)$ and 
$\eta^*\overline{(E, r, l_o, \Phi^q)}$ are gauge equivalent if and only if 
$q \in \R$. Therefore the Lemma follows from the Hitchin-Kobayashi correspondence for 
parabolic bundles.
\end{proof}

\begin{lemma}\label{trivialholo}
Let $\wt r=\frac{l}{k} \in (\frac{1}{4}, \frac{1}{2})$ be rational with coprime $l,k \in \N$ and  $r =2 \wt r-\frac{1}{2}=\frac{4l-k}{2k}$.  The  connection $\nabla^{a,\chi, r }$  (on $T_{\tau}^2$)  induces on $\Sigma_g$ 
as defined in \eqref{eqRS}
 the trivial holomorphic structure if $\chi = \tfrac{\pi }{4 \tau} \in \R$ modulo sign and lattice points.
Likewise, 
 the  connection $\nabla^{a,\chi, r }$  (on $T_{\tau}^2$)  induces on 
$\wt\Sigma_g$ \eqref{eqnRS2}
 the trivial holomorphic structure if $\chi =i \tfrac{\pi }{4} \in \R$ modulo sign and lattice points.
\end{lemma}
\begin{proof}
The proof is completely analogous to the proof of \cite[Proposition 3.1]{BDHH}. In fact it is just a reformulation of Lemma \ref{trivialoncovering} using Lemma \ref{Lemma:4-spin}.
\end{proof}
 In view of the above lemma we define
\begin{definition}\label{sec:ass}
We denote by 
 $H \in$ Jac$(T^2_\tau)$ the holomorphic structure  (or the corresponding parabolic bundle on $\mathcal S_4$)  that lifts to the trivial holomorphic structure over the associated compact Riemann surface $\Sigma$, i.e., either  $H =  \tfrac{\pi}{4\tau}$ for the sign choice $\sigma_2 = 1, \sigma_3 = \sigma_4 = -1$ and $\Sigma=\Sigma_g,$ or $H= i \tfrac{\pi}{4}$ for $\sigma_4 = 1, \sigma_2 = \sigma_3 = -1$ and  $\Sigma=\wt \Sigma_g.$\end{definition}
 
\begin{remark} \label{r:trivialholo} 
  If $H$ lies in the image of the map 
$$q  \in \R \longmapsto \chi(\nabla^q) \in \text{Jac}(T^2_\tau),$$ 
the two Lemmas \ref{lemma:real Hitchin} and \ref{trivialholo} show that 
the pull-back of the corresponding connection $\nabla^{q_{H}}$ to $\Sigma$   has trivial holomorphic structure and real monodromy representation.
\end{remark}

\section{Grafting and Spin Structures}\label{grafting}
\subsection{Complex projective structures}

Complex projective structures (or simply projective structures) on Riemann surfaces
are classical objects in the theory of Riemann surfaces, see \cite{Gu1} and the references therein. Consider an atlas $(U_\alpha,z_\alpha)_{\alpha\in\mathcal U}$ of  a Riemann surface for which all the transition functions
\[z_\beta\circ z_\alpha^{-1}(z)= \frac{a z+b}{c z+d} \quad \text{for some (constant)} \begin{pmatrix}a&b\\c&d\end{pmatrix}\in\text{GL}(2,\C)
\footnote {The coefficients $a, b, c, d$ depend on $\alpha,\beta\in\mathcal U$ and on the connected component of $U_{\alpha} \cap U_{\beta}$ where the transition function is defined.} \] are M\"obius transformations.  Such an atlas is called a (complex) projective atlas. Two projective atlases are equivalent if their union remains a projective atlas. An equivalence classes of projective atlases is a (complex) {\it projective structure}.

Naturally, the complex projective space $\CP^1$  itself is equipped with its standard  projective structure. For elliptic curves the natural projective structure is obtained by identifying it with the flat torus $\C /\Lambda$. All transitions functions are in this case translations.

On a compact surface $\Sigma$  of genus $g\geq2$ a special projective structure is provided by the  uniformization theorem. In this case, there is a global biholomorphism from the universal cover of $\Sigma$ to
Poincar\'e's upper-half plane $\mathbb H^2\subset\CP^1$ which is equivariant with respect to  a group homomorphism from the fundamental group of $\Sigma$ into the group of  PSL$(2,\R)$-valued M\"obius transformations (with image a Fuchsian subgroup in
PSL$(2,\R)$).  This map to  the hyperbolic plane $\mathbb H^2\subset\CP^1$ coincides with the developing map of the unique   hyperbolic metric (i.e. having constant  curvature $-1$) on $\Sigma$ compatible with the complex structure.
  
  In general, a complex projective structure on $\Sigma$ gives rise to a  developing map $dev$ from the  universal cover of $\Sigma$ to $\CP^1$ which is a local biholomorphism (but not necessarily a proper injective  map). This developing map is equivariant with respect to 
  a group homomorphism from the  fundamental group of $\Sigma$ into  PSL$(2,\C)$ (uniquely defined up to conjugation in the M\"obius group) which is referred to as  the {\it monodromy} of the complex  projective structure. By abuse of notation a complex projective structure is called a {\it real projective structure} if the corresponding   monodromy takes values in  PSL$(2,\R)$ (up to conjugation in PSL$(2,\C))$ \cite{Falt, Tak}.

\subsection{$\SL(2,\C)$-Opers}\label{opers}

A projective structure on a compact Riemann surface $\Sigma$ of genus $g\geq2$ can also be described using  particular flat $\SL(2,\C)$-connections, called opers. Let $\nabla$ be a flat $\SL(2,\C)$-connection on the rank two trivial smooth bundle $V=\underline\C^2\to\Sigma$ such that its induced holomorphic structure 
$\overline\partial^\nabla$ admits a holomorphic sub-line bundle $S$ of maximal degree $(g-1).$
Take a complementary $C^\infty$-bundle $S^*\subset V$ and write
\begin{equation}\label{eq:oper}\nabla=\begin{pmatrix} \nabla^{S}& \psi\\ \varphi &\nabla^{S^*}\end{pmatrix}\end{equation}
with respect to $V=S\oplus S^{*}.$ As $S$ is a holomorphic subbundle $\varphi$ is a $(1,0)$-form with values in 
Hom$(S,S^*)$. Moreover, the flatness of $\nabla$  implies that
\[\varphi\in H^0(\Sigma, K_\Sigma(S^{*})^2).\]
If $\varphi\equiv 0$,  then $S$ is a  parallel sub-line bundle of $V$ with respect to the connection $\nabla$ and, consequently,  it must have  degree zero (and not $g-1$). 
Therefore, for  $g\geq2$  the holomorphic section  $\varphi$ is not identically zero. Moreover, since the degree of $K_\Sigma(S^{*})^2$ is zero, the section $\varphi$ is nowhere vanishing. Therefore  $S$ is a  spin bundle, i.e., $S^2=K_\Sigma$ as holomorphic line bundles, and  the section $\varphi$ can be identified with the constant section $\bf 1$ of the trivial holomorphic line bundle  $\underline{\C}\rightarrow \Sigma$.

\begin{definition}
A flat $\SL(2,\C)$-connection of the form \eqref{eq:oper} on a compact Riemann surface is called an oper.
\end{definition}

Given an oper $\nabla$ on the Riemann surfaces $\Sigma$ the induced projective structure is obtained as follows. Consider, on an open simply connected subset $U\subset \Sigma$, two linear independent $\nabla$-parallel sections of $V= S \oplus S^*$ 

\[\Psi_1=\begin{pmatrix} x_1\\ y_1\end{pmatrix},\quad \Psi_2=\begin{pmatrix} x_2\\ y_2\end{pmatrix}.\]
Then $y_1$ and $y_2$ are holomorphic sections of $S^*$ as the projection $V\to V/S$ is holomorphic.  The quotient $z=y_1/y_2$ defines a holomorphic map to $\CP^1$.  Choosing two other linear independent parallel sections
\[\widetilde\Psi_1=\begin{pmatrix} \wt x_1\\ \wt y_1\end{pmatrix}= a\Psi_1+b\Psi_2 \quad  \widetilde\Psi_2=\begin{pmatrix} \wt x_2\\ \wt y_2\end{pmatrix}= c\Psi_1+d\Psi_2\]
(with $ad-bc=1)$ amounts into
\[z=y_1/y_2\longmapsto \wt z = \wt y_1/\wt y_2 =  \frac{a y_1+ b y_2}{c y_1+d y_2}=\frac{ az +b}{cz+d},\]
which is a M\"obius transformation. Because $\varphi$  is nowhere vanishing, the map $z$ is unbranched, i.e., $z$ is a local (holomorphic) diffeomorphism, and we obtain a projective atlas.

\subsection{Grafting} Grafting, or more precisely $2\pi$-grafting of the uniformization, introduced by Maskit \cite{Mas}, Hejhal \cite{Hej} and Sullivan-Thurston \cite{ST}, is a procedure to obtain infinitely many distinct real projective structures. Our short description here follows Goldman \cite{Goldman}.

Consider the real projective structure given by the uniformization (Fuchsian) representation of a Riemann surface $X$ and its developing map $dev$ to $\mathbb H^2\subset\CP^1$.
Every non-trivial element of the first fundamental group $[\gamma] \in \pi_1(X)$ can be represented by a unique geodesic $\gamma \subset X$ with respect to the constant curvature $-1$ metric on $X$ (up to orientation).
Under the developing map $\gamma$ is mapped to a circular arc.
The corresponding full circle $C$ intersects the boundary at infinity of the hyperbolic disc at two points.
The monodromy of the uniformization representation along $\gamma$ is given
by an element $A \in \SL(2,\R)$, unique up to sign.
The sign depends on the lift of the monodromy  representation from $\PSL(2,\R)$  to  $\SL(2,\R)$. 

On the other hand, a hyperbolic transformation conjugated to $A\in \SL(2,\R)$ gives 
rise to a Hopf torus $T_A$ endowed with a projective structure as follows. There exist 
two (unique) circles $S_1$ and $S_2 \subset \C P^1$ that are invariant under the 
transformation $A.$ Let $C_1$ be another circle in $\C P^1$ intersecting both $S_1$ and 
$S_2$ perpendicularly, and consider $C_2=A(C_1).$ Since $C_1$ and $C_2$ have no 
intersection points, they bound an annulus $\mathbb A$. The torus $T_A$ is then obtained 
from gluing $C_1$ and $C_2$ via $A$ and possesses by construction a projective structure. 
The monodromy of the corresponding  projective structure on $T_A$  is trivial along $C_1$, while it is $A$ along
the curve obtained from projecting the arc on $S_1$ between the intersection points with $C_1$ and $C_2=A(C_1)$.
When cutting  $T_A$ along $S_1$ (rather than $C_1$) we obtain another annulus and we 
denote this cylinder together with its two boundary components by $|T_A|.$

Grafting along the geodesic $\gamma \subset X$ glues $X$ and the appropriate Hopf torus $T_A$ for the hyperbolic $A \in \SL(2,\R)$ representing the monodromy along $\gamma$. More explicitly, let $S_1$ be the unique  circle in $ \C P^1 $ containing  the image of the geodesic $\gamma$ under the  uniformization map, and $S_2$ be the boundary of $\mathbb H^2.$ Then each of the two boundary components of $|T_A|$ can be identified with the image of the geodesic $\gamma.$ Therefore, $ X\setminus \gamma$ and $|T_A|$ can be glued to obtain a new Riemann surface $X^G$ without boundary (the notation $X^G$ stands for grafted $X$). Moreover, the induced new projective structure has the same monodromy as the  uniformization of $X$. In particular, it is a real projective structure.
Note that (the developing map $dev_G$ of) the projective structure on $X^G$ induces a curvature -1 metric away from a singularity set. The singularity set is given by the intersection of the annulus $\mathbb A$ with the circle $S_2$ considered as the boundary of the hyperbolic plane $\mathbb H^2$ consisting of two smooth 
curves that are both, considered as curves in $X^G$, closed and homotopic to $\gamma$.

Iteration leads to the construction of infinitely many Riemann surfaces with distinct real projective structures, but their monodromy remains the same uniformization monodromy of the initial Riemann surface $X$.
Altogether, starting with the isotopy class $\mathcal C$ of a simple closed curve on $X$, and applying the grafting construction along  the closed curve
yields (infinitely many) different Riemann surfaces of the same genus with real projective structures. We will refer to grafting along a simply closed geodesic $\gamma$ once as simple-grafting and neglect the possibility of grafting multiple times along the same geodesic in the following.

\begin{remark}
In what follows, $\Sigma$ will be the Riemann surface $X^G$ obtained from grafting (once) the Riemann surface $\Sigma^{uG}:=X$, where the superscript ${uG}$ stands for ungrafting.
\end{remark}

\subsection{Spin structures} \label{Section Spin}
Let $\Sigma$ be a compact Riemann surface of genus $g$.
A spin structure is a choice of a holomorphic line bundle $S$ with \[S^{\otimes 2}=K_\Sigma.\]
Two spin bundles   differ by a holomorphic line bundle which squares to  the trivial holomorphic line  bundle. If we equip this holomorphic  line bundle with its unique flat unitary connection, its monodromy takes values in $\Z_2$ (as the monodromy squares to the identity). Such a holomorphic line bundle together with its flat connection is called a $\Z_2$-bundle. It is determined by its  monodromy representation which is a group homomorphism from the fundamental group of $\Sigma$  into the abelian group $\Z_2.$  Hence the space of spin structures is an affine space with underlying translation   vector space  $H_1(\Sigma,\mathbb Z_2)$.

From a topological point of view, a spin structure is given by a quadratic form
\[Q\colon H_1(\Sigma,\mathbb Z_2)\to\mathbb Z_2,\]
whose underlying bilinear form is the intersection form (mod 2), see \cite{John, At}. The relationship between these two viewpoints can be explained as follows: Fix a given line bundle  $S$ with $S^{\otimes 2}=K_\Sigma.$ Then,  for every closed and immersed curve $\delta \colon S^1\to\Sigma$ there is a unique $\omega \in\Gamma(S^1,\delta^*K_\Sigma)$ with
\[\omega(\delta')=1.\]

Let  $\widetilde\delta\,\colon\, \mathbb R \,\longrightarrow\, \Sigma$ be the   immersed curve defined as the lift of $\delta$ to the universal covering  $\mathbb R \to S^1$.  Consider  the pull-back $\widetilde\omega$ of $\omega$  to  a (non-vanishing)  section of $\widetilde\delta^*K_\Sigma$.  Up to sign, there exists a unique   section 
\[\widetilde s\,\in\,\Gamma(\mathbb R,\,\widetilde\delta^*S)\] such that  $\widetilde{s}^2
\,=\,\widetilde\omega$. The $\mathbb Z_2$-monodromy of $S$ along $\delta$ is $1$  if $\widetilde s$ (or, equivalently, $-\widetilde s$)  is invariant by the action of the fundamental group of $S^1$ on $\mathbb R$ by deck-transformations. The $\mathbb Z_2$-monodromy of $S$ along $\delta$ is $-1$  if $\widetilde s$   is  transformed in $-\widetilde s$ by the action of the generator of the  fundamental group of $S^1$ on $\mathbb R$.  This  $\mathbb Z_2$-monodromy only depends on the class of $\delta$ in  $H_1(\Sigma,\mathbb Z_2)$. Associating to the class of $\delta$ in  $H_1(\Sigma,\mathbb Z_2)$ the above monodromy with values in $\mathbb Z_2$ uniquely determines a quadratic form  $Q\colon H_1(\Sigma,\mathbb Z_2)\to\mathbb Z_2$ whose underlying bilinear form is the intersection form (mod 2), and hence a topological
spin structure, see \cite{John} and also \cite{P} and \cite[Section 10]{Bob}.

\subsubsection{Spin structures and opers}

Let $\nabla$ be an oper on $V\to\Sigma$ given by \eqref{eq:oper} corresponding to real projective structure,
i.e., its monodromy takes values in
\[P\mathrm{SL}(2,\mathbb R)=\text{Isom}(\mathbb H^2).\]

The isomorphism between $S^{\otimes 2}$ and $K_\Sigma$  can be made explicit in two equivalent ways.

First, consider for  $p \in \Sigma$ an arbitrary $s_p\in S_p\subset V_p,$ then
\[\nabla s_p\wedge s_p\in (K_\Sigma)_p\]
is well-defined and gives rise to a bilinear map $S_p\times S_p\to (K_\Sigma)_p,$ since $S$ is a holomorphic subbundle of $V.$ This bilinear form is non-degenerate, as $\varphi$ is non-vanishing, and defines a holomorphic isomorphism  between $S^{\otimes 2}$ and $K_\Sigma.$
Likewise,  consider the (locally defined) parallel sections

\[\Psi_1=\begin{pmatrix} x_1\\ y_1\end{pmatrix},\Psi_2=\begin{pmatrix} x_2\\ y_2\end{pmatrix} \in\Gamma(U,S\oplus S^*). \]

determined by the initial condition
\[(\Psi_1)_p=\begin{pmatrix} s_p\\ 0\end{pmatrix}\quad\text{and}\quad (\Psi_2)_p=\begin{pmatrix} 0\\ t_p\end{pmatrix}\]
with
\[(t_p,s_p)=1.\]
Then, a direct computation shows
\[s_p\otimes s_p=\omega_p\in (K_\Sigma)_p\]
where
\[\omega\,=\, -d(\tfrac{y_1}{y_2}).\]

Another way to obtain the spin bundle $S$ is the following. The standard projective structure on $\mathbb CP^1$ is induced by the trivial connection $d$ on the trivial holomorphic rank 2 bundle.  Its  spin bundle is the tautological bundle $\mathcal O(-1), $ i.e,  the fiber at $l\in\CP^1$ is the line $l.$ Consider for a general Riemann surface $\Sigma$ a projective structure and developing map $dev$ induced by an oper $(V,\nabla)$.  Let $M$ be the monodromy homomorphism of $\nabla$. Then the bundle $V$ is given by the twisted bundle
\[V\,=\,(\widetilde\Sigma\times\C^2)/\sim,\]
where $(p, v) \sim (\wt p, \wt v)$ if and only if $\wt p = \gamma_* p$ and $\wt v= M_\gamma v$ for a $\gamma \in \pi_1(\Sigma).$ The spin bundle $S$
is then given by the twisted pull-back of the tautological bundle $\mathcal O(-1) \rightarrow \C P^1$ via the developing map
\begin{equation}\label{inducedspin}
S=dev^*\mathcal O(-1)/\sim.
\end{equation}

Due to its topological invariance, continuous  deformations of the Riemann surface and the oper do not change the (topological) spin structure.
(For closed curves in the moduli space of Riemann surfaces that are not null-homotopic, the deformation of the spin structure along the curve might have monodromy, see for example \cite{At}. This corresponds to a non-trivial action of the mapping class group.)

Note that the difference between two quadratic forms on $H_1(\Sigma, \mathbb Z_2)$ 
corresponding to spin structures  is given by a linear form on $H_1(\Sigma, \mathbb Z_2)$, see  \cite{John}.
\begin{lemma}\label{fundamentallemma}
Let $\Sigma$ be a compact Riemann surface.
The uniformization connection $\nabla^U$ and the simple-grafting connection $\nabla^G$, along  an isotopy class $\mathcal C$ of a simple non-null-homotopic   curve $\mathcal C$  on $\Sigma$,
 (whose monodromy lies in the same connected component of real representations) 
induce different spin structures on the Riemann surface $\Sigma.$ More precisely, the difference of the corresponding  spin structures  is determined by the linear form  on $H_1(\Sigma, \mathbb Z_2)$  obtained by   inserting a representative of the  class $\mathcal C \in H_1(\Sigma, \mathbb Z_2)$ into the intersection form mod 2 on $\Sigma$.
\end{lemma}
\begin{remark}
Clearly, the lemma generalizes to multiple graftings as well.
\end{remark}
\begin{proof}
Consider the Riemann surface  $\Sigma^{uG}=X$  (obtained by ungrafting), i.e., its uniformization connection is gauge equivalent to
 $\nabla^G$. Since the two complex structures on $\Sigma^{uG}$ and $\Sigma,$ viewed as two points in the Teichm\"uller space of genus $g$ surfaces, can be connected by a smooth curve, both uniformization connections (in the same connected component of real points in the de Rham moduli space)
 on $\Sigma^{uG}$ and $\Sigma$ induce the same topological spin structure. It remains to show that grafting once (from the uniformization oper of $\Sigma^{uG}$ to the
 oper $\nabla^G$ on $\Sigma$)
  changes this spin structure.
  
  Recall that we
 can compute the value of the  quadratic form $Q$ associated to  the spin structure on the class in $H^1(\Sigma, \mathbb Z_2)$  of an immersed curve $\delta \colon S^1\to\Sigma$ by considering the 
  lifting  $\widetilde\omega \,\in\,  \Gamma(\mathbb R,\,\widetilde \delta^*K_\Sigma)$ to the universal cover of $S^1$ of  the unique  section   $\omega\in\Gamma(S^1,\delta^*K_\Sigma)$ defined by 
$\omega(\delta')=1$  and then  considering  the $\mathbb Z_2$-monodromy defined by  a section 
 $\widetilde{s}\,\in\,\Gamma(\R,\,\widetilde \delta^*S)$ such that $\widetilde s^2
\,=\,\widetilde \omega$.
 
Along closed curves representing an element of $\pi_1(\Sigma)$ that do not intersect the (simple) grafting curve $\gamma$, the developing map does not change. Therefore, by \eqref{inducedspin} the quadratic form of the spin structure  specialized on  those 
curves remains the same.

Consider the closed curve $\delta_1$ in $\Sigma^{uG}$  which intersects the grafting curve $\gamma$ once.
The developing map along the corresponding closed curve $\delta$ on $\Sigma$  (which intersects the grafting curve $\gamma$ once) is  obtained from the developing map along $\delta_1$ by precomposing it with the circle $C_1$ along the Hopf cylinder $T_A$ \footnote{We implicitly assume that $C_1$ does not pass through $\infty$. If it does, we can replace $C_1$ by $C_2$ without altering the remaining arguments.}.
 With respect to the holomorphic 1-form $dz$ on $\mathbb CP^1\setminus\{\infty\}$ and for $C_1= \{a e^{i\theta}+b\; | \; \theta \in [0, 2 \pi]\}$ for appropriate $a>0,b\in\C$, the section 
 $\omega$ of the canonical bundle,  along $C_1$,   is  given
by \[\omega_{\mid \theta}=-i a^{-1}e^{-i\theta}dz.\]
Using \eqref{inducedspin} we can use the holomorphic section $\sqrt{dz}=(z,1)$ on $\CP^1\setminus\{\infty\}$ along $C_1$ and observe that
a  section $s$  of $S$, along $C_1$,   such that  $s^2=\omega$ is given by
\[s(\theta)=\sqrt{-i a^{-1}} e^{-\tfrac{i\theta}{2}} \sqrt{dz}.\]
Hence, specialized on  the closed curve $\delta$,  the  quadratic form associated to the spin structure  of  the oper $\nabla^G$ on $\Sigma$ differs from the quadratic form of  the spin structure induced by  the uniformization oper of $\Sigma^G$ by a $-1$ factor. This completes  the proof.
\end{proof}
\subsection{Grafting on the 1-punctured torus $T^2_\tau$}\label{sec:grafting1T} Fix the parabolic weight $r \in (0, \frac{1}{2})$ and consider the real subspace $ \mathcal M_{1,1}^r(\R)$ of 
$ \mathcal M_{1,1}^r$ corresponding to the real character variety.  By \cite[Theorem 3.4.1]{Gold}  each non-compact connected component of the real character variety   $\mathcal M_{1,1}^r(\R)$  is in one-to-one correspondence to hyperbolic structures on the one-punctured torus with conical angle $4 \pi r$ at the marked point (as the rotation angle satisfies $\theta=2r$).
Therefore, we refer to elements of $\mathcal M_{1,1}^r(\R)$ as conical hyperbolic structures for short.  For conical hyperbolic structures  \cite[Theorem 1.5.2]{Bu} shows that every free homotopy class of curves on the torus can be  represented by a simply closed geodesic. 

Fix a real representation $\rho \in \mathcal M_{1,1}^r(\R)$ and denote by $X,\,Y \in \SL(2, \R)$ its 
values along $\gamma_x,\, \gamma_y \,\in\, \pi_1( T^2_\tau \setminus \{0 \}, \, p_0),$ 
respectively (see (\ref{gamma_x1}) and  (\ref{gamma_x})). Then the corresponding conical hyperbolic structure constructed in 
\cite[Theorem 3.4.1]{Gold} is obtained by gluing the opposite edges of a particular 
hyperbolic quadrilateral $P_1P_2 P_3P_4.$ The point $P_4$  hereby is a  fixed point in the 
hyperbolic plane $\mathbb H^2$ with (elliptic) local monodromy given by $r$, while  $$P_3=X \cdot P_4, \quad 
P_2\,=\,Y \cdot X \cdot P_4 \quad \text{ and } \quad P_1\,=\,X^{-1} \cdot Y \cdot X \cdot P_4.$$

Consider now in the hyperbolic plane $\mathbb H^2$ the unique hyperbolic geodesic $c$ 
perpendicular to both geodesics generated by $P_1 P_2$ and $P_4 P_3$.  Due to its 
uniqueness, this geodesic is fixed by $Y$ (it is the axis of the loxodromic isometry $Y$). 
Moreover, since two distinct geodesics in $\mathbb H^2$ intersect at most once, $c$ does 
not contain any of the vertices $P _j$, for all $j=1,\,...,\,4$. Hence $c$ is a 
simply closed geodesic representing the free homotopy class of $\gamma_y$ and $c$ does not 
contain the conical point. The same argument shows the existence of a unique hyperbolic 
geodesic $\widetilde c$ perpendicular to both geodesics generated by $P_1P_4$ and $P_2P_3$ 
representing the free homotopy class of $\gamma_x$ that does not contain the conical 
point.

Therefore, grafting of  conical hyperbolic structures can be performed along both geodesics
representing  the free homotopy class of $\gamma_x$  and $\gamma_y$.

\begin{lemma}
The  conical hyperbolic structure on the torus is rectangular if and only if 
$$\Tr(YX)=\Tr(YX^{-1}).$$ 
\end{lemma}
\begin{proof}
The trace condition $\Tr(YX)=\Tr(YX^{-1})$ implies that the reflexion across the unique geodesic $c$, which sends the edge $P_2 P_3$ to  the edge $P_1P_4,$ induces a real symmetry of  the conical hyperbolic structure. The fix point set consists of  two disjoint circles -- the geodesic $c$ and the edge $P_2P_3$ (which is identified with the edge $P_1P_4$).  Hence the corresponding torus is rectangular (and not rhombic) and the reflexion across $c$  coincides with the reflexion  across the edge $P_1P_4$ (this reflexion corresponds to the involution $i\eta$ described in Section \ref{sect:abelianization}).
\end{proof}

Orbifold grafting is well-defined on the one-punctured torus and the grafted Riemann surface remains rectangular:

\begin{lemma}\label{tg}
Grafting the $1$-punctured  rectangular torus $T^2_\tau = \C /\Gamma$, $\Gamma = \Z \oplus i \tau \Z,$ along $\gamma_x$ or  $\gamma_y$, respectively, yields another $1$-punctured and rectangular torus $T^2_{\tau^G}$ with a different real projective structure. 

\end{lemma} 

\begin{proof}
Let $Y \in \SL(2, \R)$ denotes the monodromy along $\gamma_y$. Since $y= \Tr(Y)$ satisfies  $|y|>2$,  we get that $Y$ is a hyperbolic element in $\SL(2, \R).$
Therefore, grafting glues the annulus $|T_Y|$ to $T^2_\tau$ along $\gamma_y$. This adds a rectangle to the fundamental domain of the torus  $T^2_\tau$  along the edge $\gamma_y$. Therefore, the resulting torus remains rectangular and $\tau$ decreases. 
With the same argument grafting along $\gamma_x$ add a rectangle along the edge $\gamma_y$ and the conformal type of $\tau$ increases.
\end{proof}

On the grafted punctured torus we fix (without loss of generality) the real involution, which we again denote by $\eta,$  whose fix point set lies in the homology class $2[\gamma_x]$ and contains the singular point. Let \[\mathcal M_{\R}^r =\mathcal M_{1,1}^{r,\eta}\cap\mathcal M_{1,1}^r(\R)\subset \mathcal M_{1,1}^r\]  be the subset of  $\eta$-symmetric  $\text{SL}(2, \R)$-representations on  the one-punctured torus with local monodromy determined by the parabolic weight $r$. From \eqref{character-equation} and  Lemma \ref{tausymcon2} (see also Remark \ref{41d} and Figure \ref{gammaz}) the space $\mathcal M_{\R}^r$ has four connected components, each of them is a non-compact manifold of real dimension one.

\begin{lemma}\label{Lemma:fullcomponent}
Let $\rho \in \mathcal M_{\R}^r$.  Then there is a unique $\tau\in\R^{>0}$  such that the uniformization representation $\rho^U(\tau)$ of the 1-punctured torus  $T^2_\tau$
satisfies $\rho=\rho^U(\tau)$.
\end{lemma}

\begin{proof}
The  proof of \cite[Theorem 3.4.1]{Gold} which identifies each connected component of $\mathcal M_{1,1}^r(\R)$ with conical hyperbolic structures also implies that the one-dimensional submanifold of $\eta$-symmetric representations $\mathcal M_\R^r$ is given by those hyperbolic structures compatible with the rectangular conformal structure of $T^2_{\tau}$, i.e., $\tau\in\R^{>0}$. 
\end{proof}

The following Lemma is analogous to the main result of \cite{Tan}.
\begin{lemma}\label{teichho}
For each $r\in(0,\tfrac{1}{2})$,
simple-grafting (along $\gamma_x$ or $\gamma_y$) induces a homeomorphism of the Teichm\"uller space $\R_{>0}$ of 1-punctured rectangular tori to itself.
\end{lemma}

\begin{proof}
Without loss of generality we only consider grafting along $\gamma_y.$
The map $\tau \mapsto \tau^G$, see Lemma \ref{tg}, is smooth and moreover, $\tau^G\,<
\,\tau$. Therefore, $\tau^G \rightarrow 0$ for $\tau \rightarrow 0.$
For  surjectivity we use  the intermediate value theorem:  it remains to show $\tau^G \rightarrow \infty$ for $\tau \rightarrow \infty$. 
This follows by observing that the conformal type $\tau_Y$ of the Hopf torus $T_Y$
satisfies $\tau_Y \to\infty$ 
when  $\tau\to\infty$ and  $\tau^G = (\tfrac{1}{\tau} + \tfrac{1}{\tau_Y})^{-1}$. 
Injectivity follows from the fact that the conformal type of the Hopf tori $T_Y$ and $T_{\widetilde Y}$ satisfies $\tau_{Y} >\tau_{\widetilde Y}$ for $\tau>\widetilde \tau$.
 \end{proof}

\subsection{Spin structures and projective structures on the 1-punctured torus}
If the underlying holomorphic line bundle $L$ of a logarithmic SL$(2, \C)$-connection on the 1-punctured torus in  \eqref{abel-connection}
is not a spin bundle, then the corresponding $\gamma_\chi^\pm$ must have a zero. Thus an oper $\nabla$ on a 1-punctured torus gives rise to a spin bundle on the compact torus through the special form of the connection 
in \eqref{abel-connection-spindbar} and \eqref{abel-connection-spinpartial}. 

Moreover, the corresponding quadratic form (see Section \ref{Section Spin}) $Q$ on 
$H_1(T^2_\tau\setminus\{o\},\, \Z_2)$ is well defined on $ H_1(T^2_\tau,\, \Z_2)$, as the 
local conjugacy class around the puncture is trivial in $\Z_2$, see also \cite{John,Bob}. 
Hence, using the same arguments as in the compact case, an suitably adjusted version of 
Lemma \ref{fundamentallemma} holds on the one-punctured torus, i.e., grafting changes the 
induced spin structure.

The grafted connection $\nabla^G$ has by construction  SL$(2, \R)$-monodromy lying  in the same connected component as the orbifold uniformization connection $\nabla^U$ of $T^2_\tau.$ Moreover, the Hitchin section based at $T^2_\tau$ maps $q=0$  to $\nabla^U$ and there exist a $q_G \in \R_*$  with $\nabla^{q_G} = \nabla^G$ by $\eta$-invariance.

Recall that we have chosen the spin structure  of $\nabla^U$ to be trivial, i.e., it is represented by  $\chi(\nabla^U)=0 \in$ Jac$(T^2_\tau).$
For the grafted connection $\nabla^G$ it turns out, using abelianization,   that the underlying line bundle $L$ must be a particular spin bundle, i.e., $\chi(\nabla^G)$ is a half lattice point of Jac$(T^2_\tau)$.

\begin{lemma}\label{fundamentallemmatorus}
Via abelianization
 the spin structure of simple-grafting along $\gamma_y$ is given by
 $\chi(\nabla^{G,y})=  i\tfrac{\pi}{2}$ (up to lattice points and sign),
 and  the spin structure of simple-grafting along $\gamma_x$ is given by
 $\chi(\nabla^{G,x})=  \tfrac{\pi}{2\tau}$ (up to lattice points and sign),
\end{lemma}

\begin{proof}
The unitary line bundle connection on the induced spin bundle for $\chi(\nabla^{G,y})$ is
$$d+ i\tfrac{\pi}{2} d\overline{w}+i\tfrac{\pi}{2} d w$$
which has monodromy $-1$ along $\gamma_x$ and monodromy $1$ along $\gamma_y$. For $\chi(\nabla^{G,x})$ it is
$$d+ \tfrac{\pi}{2\tau} d\overline{w}-\tfrac{\pi}{2\tau} d w$$
 and 
has monodromy $-1$ along $\gamma_y$ and monodromy $1$ along $\gamma_x$.
Thus, the statement follows by the same arguments as Lemma \ref{fundamentallemma}, since after grafting along 
$\gamma_y$ the spin structure change by multiplication with $1$  along $\gamma_y$ (meaning unchanged)  and  by multiplication with $-1$ along $\gamma_x$. It is the other way around  when  grafting along $\gamma_x.$
\end{proof}

\section{Holomorphic connections with $\SL(2, \R)$-monodromy}\label{holo}

Consider the natural  2-fold covering map $\pi\colon Jac(T^2_\tau)\to\CP^1$ branched at the four spin bundles.
The real subspace provided by Lemma \ref{tausymcon} is  mapped to  $\mathbb R\subset \C $. In fact, as the corresponding elliptic curve $T^2_\tau$ being rectangular,  its associated  $\wp$-function maps the   four  spin bundles   in the Jacobian (which identifies with the half periods and the critical points of $\wp$)  to the real axis.

\begin{lemma}\label{onto}
The map $h \colon  \R \rightarrow  \R \subset  \CP^1=\pi(\text{Jac}(T^2_\tau))$ given by $q \mapsto\pi( \chi(\nabla^{q}))$ is continuous. Moreover, there
exists $q_H \in \R\setminus\{0\}$ satisfying either
$\chi (\nabla^{q_H})=\frac{\pi}{4\tau}$ or $\chi (\nabla^{q_H}) = i\frac{\pi}{4}$ (up to sign and adding lattice points). 
\end{lemma}

\begin{proof}
The holomorphic structure $\chi(\nabla^q)$ as a map into Jac$(T^2_\tau)$ is only well-defined up to sign.  The projection to $\C P^1$ removes this multivaluedness, and  $h$ is well-defined and continuous.
It should be noted that this is, up to normalisation, the Tu-invariant of  $\nabla^{q}$ in \cite{Lo}.
By $\eta$-invariance, i.e., Lemma \ref{lemma:real Hitchin} and Lemma \ref{tausymcon}, and using  the fact that $\pi$ maps
exactly the (real) $\eta$-invariant points of the Jacobian (i.e., $\chi$ lying in the lines given by Lemma \ref{tausymcon}) to real points (including $z=\infty$) in $\CP^1$, we obtain that $h\colon \R\to \R P^1\subset\CP^1.$

Grafting along $\gamma_y$ (or $\gamma_x$) changes the spin structure by Lemma \ref{fundamentallemmatorus}. 
The map $h$ sends the point $q=0$  to $0 \in$ Jac$(T_{\tau}^2)$ and the point
$q=q_G$ to $i\tfrac{\pi}{2} \in $Jac$(T_{\tau}^2)$ (or  to $\tfrac{\pi}{2\tau} \in $Jac$(T_{\tau}^2)$). 

By continuity 
the image of $h$ contains either the interval $\pi([0, i \frac{\pi}{2}])\subset\R P^1,$ where  $$[0, i\tfrac{\pi}{2}]:=\{ ti\tfrac{\pi}{2}\mid 0\leq t\leq 1\}$$ and analogously for the subsequent intervals, or the 
closure of its complement  (inside $\R P^1$) given by
\[\R P^1\setminus \pi((0, i \tfrac{\pi}{2}))=
\pi([0, \tfrac{\pi}{2\tau}])\cup\pi([\tfrac{\pi}{2\tau}, \tfrac{\pi}{2\tau} +i \tfrac{\pi}{2} ])\cup\pi([\tfrac{\pi}{2\tau}+i\tfrac{\pi}{2},i ]). \]

Consequently, there exists $q_H$ (and a second $\widetilde q_H$ for grafting along $\gamma_x$) in the interval between $q=0$  and $q= q_G$ such that $\chi(\nabla^{q_H})=\frac{\pi}{4\tau}$  or $\chi(\nabla^{q_H})= i\frac{\pi}{4}$ (up to sign and adding lattice points). 
\end{proof}

\begin{remark}
The Lemma shows that the holomorphic structure $H$ (see Definition \ref{sec:ass}) which lifts to the trivial holomorphic structure on the associated Riemann surface $\Sigma$ is attained whenever the weight $r$ is rational. 
Since the map $h$ depends continuously in $\tau$, the bundle type $H$ does not change for a continuous deformation of $\tau.$
\end{remark}

Lemma \ref{trivialholo} together with Remark \ref{r:trivialholo}  gives the following theorem.
\begin{theorem}\label{holoconnection}
For every rational weight $\wt{r}= \tfrac{l}{k} \in (\frac{1}{4}, \frac{1}{2})$ and every $\tau\in \R_{>0}$,
the Riemann surface $\Sigma(\tau)$ given by the $g+1$-fold covering \eqref{eqRS} or \eqref{eqnRS2} (with $g=k-1$ for k odd and $g=k/2-1$ for $k$ even) of the four punctured sphere  $\mathcal S_4(\tau)$ (as in (\ref{4ps})) admits a holomorphic connection on the trivial bundle $\underline{\C}^2 \rightarrow \Sigma(\tau)$ with SL$(2, \R)$-monodromy.
\end{theorem}
\begin{remark}
Restricting to rational weights $r=\tfrac{k-1}{2k}$ we obtain a new proof of the main theorem of \cite{BDHH}, without applying the WKB analysis. In particular, we obtain holomorphic systems on compact Riemann surfaces with Fuchsian
representations that are not far out in the Betti moduli space of real representations.
\end{remark}

\begin{corollary}\label{ordering}
The triple $\nabla^U,$ $\nabla^H$ and $\nabla^G$ is ordered in the character variety, i.e., either 
$$y(\nabla^U) < y(\nabla^H) < y(\nabla^G) \quad 
\text{ or }
\quad y(\nabla^U) > y(\nabla^H) > y(\nabla^G).$$
\end{corollary}
\begin{proof}
From \eqref{extraaffineeq},
each connected component of  $\mathcal M_{\R}^r$ admits $y= \Tr(Y)$ as  a global coordinate, where $Y$ is defined as the monodromy along $\gamma_y$ (see (\ref{gamma_x})). Moreover, the image of $y$ is either $(2,\infty)$ or $(-\infty,-2).$ 

Hence,  for any choice of spin structure on the torus
the map \[q \in \R \rightarrow  y(\nabla^q) \in \R\]
 is a diffeomorphism onto the corresponding connected component given by  $y \in (2,\infty)$ or $y\in (-\infty,-2).$ 
Therefore, the map $y(\nabla^q)$ is strictly monotonic in $q \in \R$ and the Corollary follows from $q_H$ lying between $q=0$ and $q= q_G.$
\end{proof}

\begin{remark}\label{rempm}
With the above notations,
since  $q\in\R$ (by Hitchin-Kobayashi correspondence) and $y\in(2,\infty)$  (the trace)  are both global coordinates
on every connected component of $\mathcal M^r_\R,$ the space of $\eta$-invariant conical hyperbolic structures, we have that $y(\nabla^q)$ is either strictly monotonically increasing or decreasing.  
Moreover, the conformal type for grafting $n$-times along $\gamma_x$ and $n$-times along $\gamma_y$ degenerate for $n \rightarrow \infty$ to the two different ends of $\mathcal M^r_\R$. Therefore,  either
\[y(\nabla^{q_{G,x}})<y(\nabla^{q=0})< y(\nabla^{q_{G,y}}) \quad \text{ or } \quad  y(\nabla^{q_{G,x}})>y(\nabla^{q=0})> y(\nabla^{q_{G,y}}).\]

Since the $y$-coordinate of connection $\nabla^{q_H}$ 
lies between $y(\nabla^{q_{G,x}})$ and $y(\nabla^{q=0})$,  or respectively,  $y(\nabla^{q_H})$ lies between $y(\nabla^{q_{G,y}})$ and $y(\nabla^{q=0})$, we reverse the ordering in Corollary \ref{ordering}  by grafting along $\gamma_x$ instead of $\gamma_y$.
\end{remark}

\begin{lemma}\label{localdiffeo}
Consider the  family  of 1-punctured rectangular tori $T^2_\tau$, for $\tau \in \R_{>0}$. Then the map defined using the abelianization coordinates  
$$\mathfrak M\colon \R^2 \ni (a, \tau) \rightarrow \mathcal M_{1,1}^{r,\eta}, \quad (a, \tau) \mapsto \rho(\nabla^{a, \chi_0}(\tau))$$
is a local diffeomorphism for fixed $\chi_0 = \tfrac{\pi}{4\tau}$ away from $a_0=-\tfrac{\pi}{4\tau}.$
\end{lemma}
\begin{remark}\label{imaginary_iso_def}
An analogue Lemma also holds when choosing $\chi_0 = i\tfrac{\pi}{4}$ and $a \in i \R$ away from $a_0=i\tfrac{\pi}{4}.$
\end{remark}

\begin{proof}
Consider the map
$$\wt{\mathfrak M_\C}\colon \C^3 \ni (a, \chi, \tau) \rightarrow \mathcal M_{1,1}^r, \quad (a,\chi,  \tau) \mapsto \rho(\nabla^{a, \chi}(\tau))$$ 
 from $\C^3$ to the complex 2-dimensional space $\mathcal M^r_{1,1}.$ By the Riemann-Hilbert correspondence and the fact that, for fixed  conformal type $\tau$, the variables  $(a, \chi)$ are coordinates of $\mathcal M_{1,1}^r$ (away from half lattice points of $\chi$),
the kernel of the differential is only $1$-dimensional, and can be computed using the differential of isomonodromic deformations.

To show that $\mathfrak M$ is a local diffeomorphism, it suffices to prove that  
$$\mathfrak M_\C\colon \C^2 \ni (a, \tau) \rightarrow \mathcal M_{1,1}^r, \quad (a, \tau) \mapsto \rho(\nabla^{a, \chi_0}(\tau))$$ 

is an immersion at all points  $(a,\tau) \in \C^2$, with   $a \neq -\tfrac{\pi}{4\tau}$ (i.e., without restricting to the $\eta$-invariant subspaces). Indeed, Lemma \ref{tausymcon} then implies that the image of $\mathfrak M_{\C}$ restricted to real points $(a, \tau) \in \R^2$ lies  in the real 2-dimensional submanifold $\mathcal M_{1,1}^{r,\eta} \subset M_{1,1}^r$.

Since $\mathfrak M_\C$ is obtained from $\wt{\mathfrak M_\C}$ by fixing the $\chi$-coordinate, it suffices to show the kernel of $d\wt{\mathfrak M_\C}$ is transversal to the slice $\chi = \chi_0.$ The proof uses a result by  Loray \cite{Lo}. Be aware that the parameter $\tau$ used here is the conformal type of the torus $T^2_\tau,$ while the parameter $t \in \C P^1$ in \cite{Lo} is the conformal type of the $4$-punctured sphere with punctures $\{0, 1, \infty, t\}$. The transformation from $\tau$ to $t \in \C P^1$ is  a local diffeomorphism since the conformal type of the torus is non-degenerated.

First note that, for every fixed $\tau$, the parameters $\chi_0(\tau) = 
\tfrac{\pi}{4\tau}$ and $a_0(\tau)=-\tfrac{\pi}{4\tau}$ define a reducible connection on 
the 4-punctured torus $\Sigma_2(\tau) = \C/(2\Z \oplus 2 \tau i\Z)$ corresponding to a 
specific fixed irreducible representation of the 1-punctured torus $T^2_\tau$. This reducible connection is 
the tensor product of the pull-back of $D$ from the 4-punctured sphere to the 
four-punctured torus with a $\Z_2$ line bundle connection. Therefore, when varying 
$\tau$, these reducible connections define an isomonodromic deformation.

For isomonodromic deformations, the bundle type $\widetilde q(t)$ (see\cite[Corollary 10 and 
Section 5.8.4]{Lo}), determined by $\pi(\chi)$ (referred to as the Tu-invariant in 
\cite{Lo}) solves the Painlev\'e VI equation in $t$ away from four values of $\wt q(t)$ 
\footnote{The Painlev\'e VI equation is a second order ODE, and the four exceptional 
points corresponds to $\chi$ being a half lattice point of Jac$(T^2_\tau)$}. Therefore the 
flat connections (as a family in $t$) are determined by the initial value $\widetilde q(t_0)$ 
and the initial direction $\tfrac{d}{dt}|_{t=t_0} \widetilde q$ (provided by $p$; see 
\cite[Theorem 8]{Lo}).

Hence the $t$-family of connections corresponding to $\chi_0(\tau)= \tfrac{\pi}{4\tau}$ and  $a_0(\tau)= -\tfrac{\pi}{4\tau}$  determines a  solution $\wt q_{H}(t)$ of the second order 
Painlev\'e  VI equation.  Note that though this family of  connections are reducible on the 4-punctured torus, they are irreducible on the 1-punctured $T^2_\tau$ on which the Theorem of Loray holds. 
Let $\wt q(t)$ be another solution for a different isomonodromic deformation at $t_0$ 
(where $t_0$ corresponds to $\tau_0$)
with initial value $\wt q (t_0)\,=\, \wt q_H(t_0).$ Then $\wt Q= \wt q - \wt q_H$ must have a simple zero at $t_0,$ i.e., $\tfrac{d}{dt}|_{t= t_0} \wt Q \neq 0,$ since the connections are determined by $\widetilde q(t_0)$ and  the initial direction $\tfrac{d}{dt}|_{t=t_0}
\widetilde q$ \cite[Theorem 8]{Lo}.

Since $\widetilde q\,=\,\pi(\chi)$,  we have that $\chi -\chi_0$ corresponding to the isomonodromic deformation obtained by $\widetilde q$ has simple zero at $\tau_0$,  implying transversality away from $a=-\tfrac{\pi}{4\tau}$. 
\end{proof}

\subsection{Proof of the main Theorem}\label{endgame}

Let  $\rho$ be the given RSR-representation. It defines a point $\widetilde\rho$ (in fact there are 4 preimages) in the character variety of  $\eta$-symmetric  and real representations $\mathcal M_{\R}^r$ on the 1-punctured torus. The aim is to show that there exist a $\tau \in \R_{>0}$  such that  $\widetilde \rho$ can be realized as the monodromy representation of a logarithmic  connection $\nabla^H(\tau)$  on 
the parabolic bundle over $T^2_\tau$ determined by $H,$ see Definition \ref{sec:ass},
 corresponding to the trivial holomorphic bundle on the covering $\Sigma$.

Recall that the trace $y = \Tr (Y)$ of the monodromy $Y$ along $\gamma_y$  is a global coordinate on each connected component 
of $\mathcal M_{\R}^r$. Consider the coordinate $y(\widetilde \rho)$  of our given element. 
Without loss of generality we restrict in the following to the connected component with $y (\rho)>2$ and $x(\rho)>2$.

Choose a representation $\rho_0\in \mathcal M_{\R}^r$ with $y(\rho_0)< y(\widetilde\rho)$.
By Section \ref{sec:grafting1T}, there is a (rectangular) conformal type $\tau^{uG}_0>0$ of the 1-punctured torus such that the orbifold uniformization connection $\nabla^U (\tau^{uG}_0)$ has monodromy representation $\rho_0.$
Grafting  $\nabla^U (\tau^{uG}_0)$ once along $\gamma_y$ yields a new projective structure given by a connection 
$\nabla^G(\tau_0),$ where $\tau_0>0$ is the rectangular conformal type obtained from grafting $T^2_{\tau^{uG}_0}$.

Assume further without loss of generality
that we have $y(\nabla^U(\tau_0))  < y(\nabla^G(\tau_0))$ at $\tau_0$. If this assumption does not hold for simple-grafting
 along $\gamma_y$, it holds for simple-grafting along $\gamma_x$, and the remainder of the proof works with the same arguments, see Remark \ref{rempm}. Then there exists a $\nabla^H(\tau_0) = \nabla^{q_H}(\tau_0),$ see  Lemma \ref{onto},  on the parabolic bundle determined by $H\in$ Jac$(T^2_{\tau_0})$ with
\begin{equation}\label{eqin}
y(\nabla^U(\tau_0)) \,<\,y(\nabla^H(\tau_0)) \;<\, y(\nabla^G(\tau_0))
\,=\, y(\rho_0)\,<\, y(\widetilde\rho).\end{equation}

Since $y$ is a global coordinate on each connected component of $\mathcal M_{\R}^r,$ the above  inequality remains true for continuous deformations of the connections
induced by a continuous deformation of the conformal type $\tau$. Since $\nabla^U(\tau)$ and $\nabla^G(\tau)$ depends continuously on $\tau$ by Lemma \ref{Lemma:fullcomponent} and Lemma \ref{teichho}, the inequality holds for a deformation in $t\mapsto\tau(t)$, if there exist a corresponding continuous deformation
\[t\mapsto\nabla^H(a(t),\tau(t))\] with prescribed underlying parabolic structure $H(\tau(t))$, see the definition of $\mathfrak M$ in Lemma \ref{localdiffeo}. 
Locally such a deformation of  $\nabla^H(\tau_0)$ exists, since $\mathcal M^r_\R$ is a real 1-dimensional submanifold of $\mathcal M_{1,1}^{r,\eta},$ on which  $(a, \tau)$ are local  coordinates by Lemma \ref{localdiffeo}.

In the following, we use \[t=y\]
as deformation parameter, i.e., the family $t\mapsto \nabla^H(a(t),\tau(t))$
of logarithmic connections on the prescribed parabolic bundle H over the 1-punctured torus given by $\tau(t)$
satisfies
\[y(\nabla^H(a(t), \tau(t)))=t.\]
We call such a  continuous (and therefore smooth) family {\em admissible}
if additionally  the following condition hold
\[ \tau(t_0)=\tau_0 \quad\text{and}\quad \nabla^H(a(t_0),\tau(t_0))=\nabla^H(\tau_0),\]
 for $t_0=y_0:=y(\nabla^H(\tau_0)).$ The aim is to show
that 
\[y_{max}:=\text{sup} \left\{y\in\R_{\geq y_0}\mid \exists \text{ an admissible
family } t \mapsto \nabla^H(a(t),\tau(t)) \text{ with } t\in  [y_0,y) \right\}\]
satisfies
\[y_{max}= \infty> y(\widetilde \rho).\]

By construction we therefore have
\begin{equation}\label{eqin2}
y(\nabla^U(\tau(t))) <y(\nabla^H(t,\tau(t))) \;< y(\nabla^G(\tau(t))).\end{equation}
 for all $t\in[y_0,y_{max}).$ Let us assume that $y_{max}<\infty.$  Since $\mathcal M_\R^r$ satisfies \eqref{extraaffineeq} we obtain $x_{max}\in (2,\infty)$ is finite as well. Moreover,  the corresponding conformal type
 \[\tau_{sup}:=\limsup_{t\to y_{max},t<y_{max}} \tau(t)\in\R^{>0} < \infty\]
is finite  by \eqref{eqin2}  using Lemma \ref{teichho}).

Next, by definition of $\nabla^H(a(t),\tau(t))$ 
together with the fact that $(a,\chi)$ are coordinates on the moduli space,
 there is either a
real  or purely imaginary function $a(t),$ see Lemma \ref{tausymcon}, such that
\[\nabla^H(t,\tau(t))=\nabla^{a(t),H,r}\]
on the 1-punctured torus $T^2_{\tau(t)}$ is real and $\eta$-invariant. Assume that 
\[a_{sup}:=\sup_{t\to y_{max},t<y_{max}}|a(t)|=\infty.\]
Using the WKB analysis of Mochizuki in \cite[Appendix]{BDHH} along $\gamma_x$ if $a(t)$ is real, or along  $\gamma_y$ if $a(t)$ is purely imaginary, with respect to the diagonal Higgs field

$$\Phi_t = 
a(t)  \begin{pmatrix} dw & 0\\0&-dw\end{pmatrix}$$

 we  obtain that, up to taking a suitable subsequence for which the
 conformal type converges to $\tau_{sup}$, either 
\[x(\nabla^H(a(t),\tau(t))) \rightarrow\infty\quad  \text{ or }\quad y(\nabla^H(a(t),\tau(t))) \rightarrow \infty \]
which is a contradiction. Therefore $y_{max} = \infty$ concluding the proof.
\qed

\section{The dodecahedral example}

It is well-know that there
exists exactly 4 compact, regular, and space-filling tessellations of the hyperbolic 3-space. They were first described by Coxeter in \cite{Co1,Co2}. Here we are interested in
the order-4 dodecahedral honeycomb with Schlaefli symbol $\{5,3,4\}$. The fundamental domain of this tessellation is the regular  dodecahedron with  dihedral angle $\pi/2$. There are four dodecahedra around each edge  and eight dodecahedra around each vertex in an octahedral arrangement. Note that, in contrary to the case of dihedral angle $2 \pi/5$ which leads to the construction of  a Seifert-Weber hyperbolic 3-manifold,   the  cocompact discrete subgroup of $\text{PSL}(2, \C)$, constructed by identification of opposite pairs of faces of  the regular dodecahedron with dihedral angle  $\pi/2$,  admits  nontrivial  torsion. Though the  $\pi/2$-dodecahedron is not  a  fundamental domain for  any discrete  torsion-free subgroup in $\text{PSL}(2, \C)$, there are discrete groups  with torsion in $\text{PSL}(2, \C)$ which admit  the $\pi/2$-dodecahedron  as  its fundamental domain.  Groups containing torsion-free subgroups of finite (small) index  can be determined using the Reidemeister-Schreier method \cite{Be}.

A barycentric subdivision of the regular dodecahedron cuts it into $120$ copies of its characteristic cell. In  our case where the dihedral angle is  $\pi/2$, the characteristic cell is the  tetrahedron $\mathbb T$  with dihedral angles 
\[\tfrac{\pi}{2},\tfrac{\pi}{2},\tfrac{\pi}{2},\tfrac{\pi}{5},\tfrac{\pi}{3},\tfrac{\pi}{4}.\] 

\begin{figure}[h]
\includegraphics[scale=0.12]{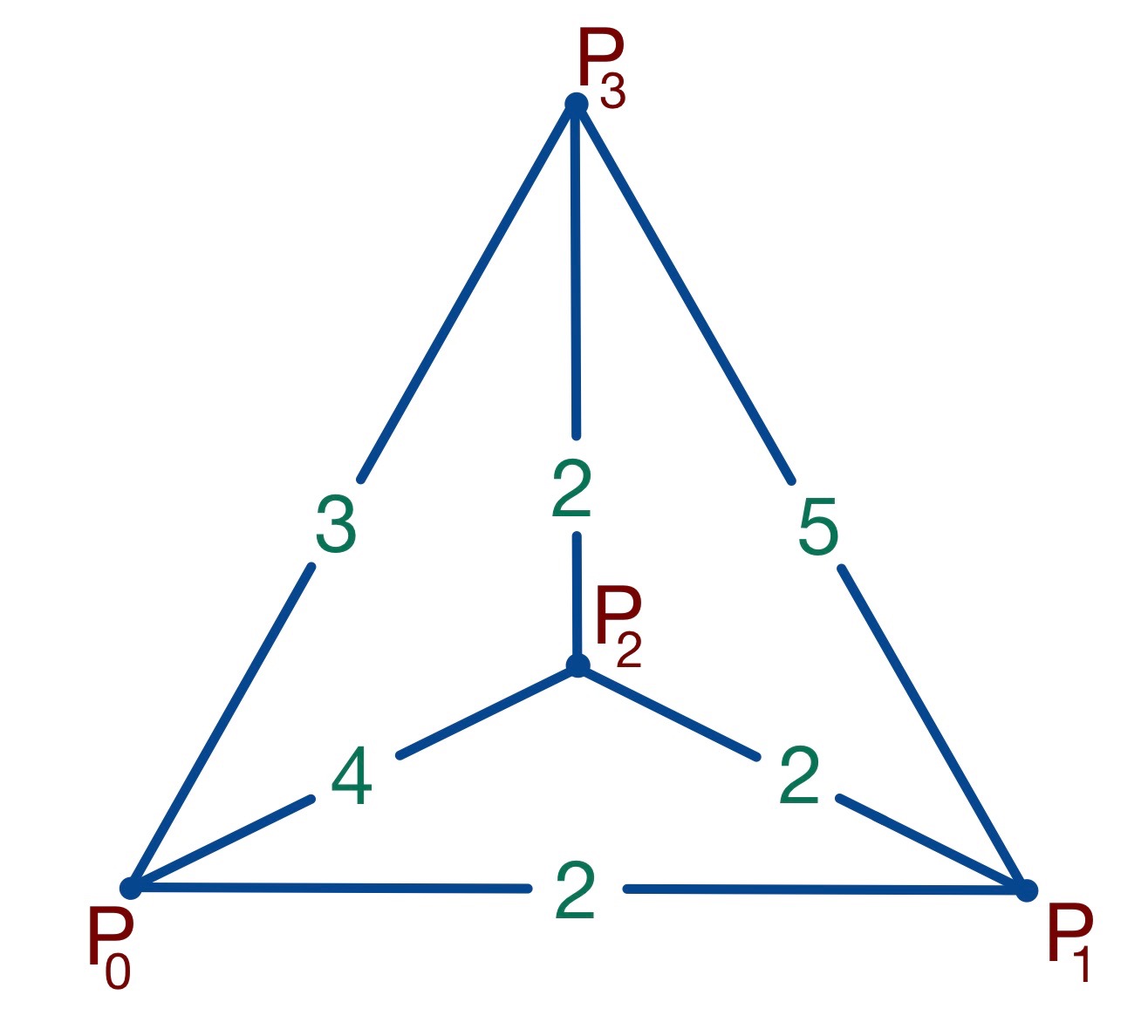}
\caption{The fundamental $(5,3,4)$ tetrahedron.}\label{fig:Tet}
\end{figure}

The aim is to use $\mathbb T$ to explicitly writing down a RSR-representation $\rho$ that is compatible with the cocompact lattice given by the dodecahedral tiling of $\mathbb H^3.$
For computational convenience,   we use the hyperboloid model of the hyperbolic 3-space. To fix notations, consider the Lorentzian space
 \[\R^{3,1}=\left\{h=\begin{pmatrix} x_0+x_1&x_2+i\, x_3\\x_2-i\, x_3&x_0-x_1\end{pmatrix}
\,\mid\,\,  \overline{h}^T\,=\,h \right \}\]
equipped with its  canonical indefinite inner product $(.,.)$ of signature $(1,3).$  The quadratic form of $(.,.)$ is 
the negative of the determinant.

Then the hyperbolic 3-space is given by
\begin{equation}
\begin{split}
\mathbb H^3&=\{(x_0,x_1,x_2,x_3)\mid-x_0^2+x_1^2+x_2^2 +x_3^2=-1,\; x_0>0\}\\
&=\left\{h=\begin{pmatrix} x_0+x_1&x_2+i\, x_3\\x_2-i\, x_3&x_0-x_1\end{pmatrix} \mid \; \text{det}(h)=1; \,  \overline{h}^T=h; \, \text{Tr}(h)>0\right\}
\end{split}
\end{equation}
endowed with the Riemannian metric induced by  the quadratic form $q.$ The submanifold $$\mathbb H^2 \,=\{(x_0,x_1,x_2,0 )\mid-x_0^2+x_1^2+x_2^2=-1,\; x_0>0\}$$ endowed 
with the induced Riemannian metric is totally geodesic and a copy of the hyperbolic 2-plane.

Recall that there exist a natural double covering of the connected component of the identity of the isometry group of $\mathbb H^3$ by SL$(2, \C)$

\[\mathrm{SL}(2,\C)\to\mathrm{SO}(3,1)\subset \mathrm{Isom}(\mathbb H^3)\]
 given by the right action
\[(h,g)\in \mathbb H^3 \times\mathrm{SL}(2,\C)\longmapsto h\cdot g= \bar g^Thg\in\mathbb H^3.\]

Restricting this map to the real subgroup $\mathrm{SL}(2,\R) \subset  \mathrm{SL}(2,\C)$ preserves $\mathbb H^2 \subset \mathbb H^3$ and defines a double covering
$$\mathrm{SL}(2,\R)\to\mathrm{SO}(2,1)\subset \mathrm{Isom}(\mathbb H^2).$$

Consider  the hyperbolic   tetrahedron in $\mathbb H^3$ defined by the following  4 vertices \footnote{We choose $\frac{1}{5^{1/4}}\in\R^{>0}$, and also $\sqrt{x}>0$ for all $x>0$.}
\begin{equation*}
\begin{split}
P_0&=(1,0,0,0)\\
 P_1&= \left(\sqrt{1+\tfrac{2}{\sqrt{5}}},-\frac{1}{5^{1/4}},-\frac{1}{5^{1/4}},0 \right)\\
P_2&=\left(\sqrt{\tfrac{1}{2}(3+\sqrt{5})},-\sqrt{\tfrac{2}{-1+\sqrt{5}}},0,0 \right)\\
P_3&=\left(\;\tfrac{1}{2}\sqrt{7+3\sqrt{5}},-\tfrac{1}{2}\sqrt{1+\sqrt{5}},-\tfrac{1}{2}\sqrt{1+\sqrt{5}},\tfrac{1}{2}\sqrt{1+\sqrt{5}}\;  \right)
\end{split}
\end{equation*}
 with
(geodesic triangle) faces 
\begin{equation*}
\begin{split}
&T_0\subset\text{span}\{P_1,P_2,P_3\}\cap \mathbb H^3\\
&T_1\subset\text{span}\{P_0,P_2,P_3\}\cap \mathbb H^3\\
&T_2\subset\text{span}\{P_0,P_1,P_3\}\cap \mathbb H^3\\
&T_3 \subset\text{span}\{P_0,P_1,P_2\}\cap \mathbb H^3\subset \mathbb H^2.\\
\end{split}
\end{equation*}
Each face has  a unit length normal, unique up to sign, given by
\begin{equation*}
\begin{split}
L_0&=\left(\,\tfrac{1}{\sqrt{1+\sqrt{5}}},-\tfrac{1}{2}\sqrt{3+\sqrt{5}},0,0\, \right)\\
L_1&=\left(\,0,0,\tfrac{1}{\sqrt{2}},\tfrac{1}{\sqrt{2}}\, \right)\\
L_2&=\left(\,0,-\tfrac{1}{\sqrt{2}},\tfrac{1}{\sqrt{2}},0\, \right)\\
L_3&=(\,0,0,0,1\,).\\
\end{split}
\end{equation*}
Computing the angles between $L _k,$\, $k=0,\,...,\,3$ the tetrahedron constructed here is in fact $\mathbb T$, see Figure \ref{fig:Tet}.  The corresponding Coexter group 
is generated by the reflections $R_k$ across the faces $T_k$ for $k\,=\,0,...,\,3,$ i.e., 
\[R_k(v)=v-2(v,L_k)L_k, \quad v\in \R^{3,1}.\]
Consider  its order 2 subgroup  $\mathfrak{G}$ consisting of orientation preserving transformations generated by
\[G_{m,n}:=R_mR_n, \quad 0\leq m<n\leq 3.\]

We denote by $\Gamma$ the subgroup of SL$(2, \C)$  given by the preimage of $\mathfrak{G}$ through  the above double cover.
Its generators are determined by choosing lifts $g_{k,l}$ of $G_{k,l}$ 
\begin{equation}
\begin{split}
g_{0,1}&:=\begin{pmatrix} 0&-\tfrac{1+i}{4}\left(1+\sqrt{5}+\sqrt{2(-1+\sqrt{5})}\right)\\
\frac{(2-2i)}{(1+\sqrt{5}+\sqrt{2(-1+\sqrt{5})}}&0\end{pmatrix}\\
&\\
g_{0,2}&:=
\begin{pmatrix} -\tfrac{1}{2}\sqrt{1+\sqrt{5}-\sqrt{2(1+\sqrt{5})}}& \left( -\sqrt{1+\sqrt{5}-\sqrt{2(1+\sqrt{5})}}\right)^{-1} \\
\tfrac{1}{2}\sqrt{1+\sqrt{5}-\sqrt{2(1+\sqrt{5})}}& \left(-\sqrt{1+\sqrt{5}-\sqrt{2(1+\sqrt{5})}}\right)^{-1}\end{pmatrix}\\
&\\
g_{0,3}&:=\begin{pmatrix} 0&-i\sqrt{\tfrac{1}{2}(1+\sqrt{5}+\sqrt{2(1+\sqrt{5})})}\\\frac{-i}{\sqrt{\tfrac{1}{2}(1+\sqrt{5}+\sqrt{2(1+\sqrt{5})})}}&0\end{pmatrix}\\
&\\
g_{1,2}&:=\frac{1}{2}\begin{pmatrix} -1+i&1+i\\-1+i&-1-i\end{pmatrix}\\
&\\
g_{1,3}&:=\frac{1}{\sqrt{2}}\begin{pmatrix} -1-i&0\\0&-1+i\end{pmatrix}\\
&\\
g_{2,3}&:=\frac{-i}{\sqrt{2}}\begin{pmatrix} 1&1\\1&-1\end{pmatrix}.\\
\end{split}
\end{equation}

Note that all $g_{k,l}$ are of finite order, e.g., $g_{0,2}$ is of order $5$. Consider
\[j_0:=g_{1,2}g_{1,3}=\frac{1}{\sqrt{2}}\begin{pmatrix}1&-1\\1&1\end{pmatrix} \in \text{SL}(2, \R)\]
and
\[J_1:=-(g_{0,2})^2  \in \text{SL}(2, \R),\] and
define
\begin{equation}
\begin{split}
J_2&:=j_0J_1j_0^{-1}\\
J_3&:=j_0J_2j_0^{-1}\\
J_4&:=j_0J_3j_0^{-1}.\\
\end{split}
\end{equation}
Note that $j_0$ is of order 8 and $J_1,\dots,J_4$ are all of order 10. A direct computation then shows
\[J_4J_3J_2J_1=\text{Id}.\]

\begin{theorem}\label{thm9.1}
The representation $\rho \,\colon\, \pi_1(S_4,\,s_0)\,\longrightarrow\, \text{SL}(2, \C)$ of the
4-punctured sphere $S_4$ given by
\[\rho(\gamma_k)\,=\,J_k\]
is a genus $4$ RSR-representation compatible with the cocompact  lattice $\Gamma \subset \mathrm{SL}(2, \C).$
\end{theorem}
\begin{proof}
The group $\Gamma$ (the preimage of $\mathfrak G$) is a cocompact lattice of $\mathrm{SL}(2,\C)$, since
it corresponds  (up to a reflection) to the tessellation
of $\mathbb H^3$ by copies of the tetrahedron $\mathbb T$. Since $J_k$ lies in SL$(2, \R)$ for all $k \,=\, 1,\,\cdots,\, 4$, and 
\[\text{tr}(J_2J_1)=\text{tr}(J_3J_2)=-1-\sqrt{5}<-2, \quad \text{tr}(J_3J_1)=-\tfrac{3}{2}(1+\sqrt{5})<-2,\]
 the image of the representation $\rho$ defines a {\bf real} lattice $\widehat\Gamma$ in $\mathrm{SL}(2,\R)$. Moreover, $\widehat\Gamma$ is {\bf symmetric} with $\wt r= \tfrac{3}{10} \in (\tfrac{1}{4}, \tfrac{1}{2})$, since
\[\text{tr}(J_k)=\frac{1}{2}(1-\sqrt{5})=2\cos\left(2 \pi\tfrac{3}{10} \right) \quad \text{for}  \quad k=1,...,4,\]

and {\bf rectangular} by
\[ \text{tr}(J_3J_1)\,=\,\text{tr}(J_4J_2)\,=\,-\tfrac{3}{2}(1+\sqrt{5}).\]
Since the order of the $J _j$ is 10, the genus of the representation is $4$ by Definition \ref{RSR-genus-def}.
\end{proof}

Applying Theorem \ref{maintheorem} we obtain:
\begin{corollary}
There exist a cocompact lattice $\Gamma$ in $\mathrm{SL}(2,\C)$ and a compact Riemann surface $\Sigma$ of genus $g=4$
admitting a holomorphic map $f\colon \Sigma\to \mathrm{SL}(2,\C)/\Gamma$ which does not factor through a curve of lower genus.
\end{corollary}

\subsection{Outlook: CMC 1 surfaces in hyperbolic 3-manifolds}
Surfaces with constant mean curvature $H$  in Riemannian  3-manifolds  are the critical points of the area functional with fixed enclosed volume.
By the Lawson correspondence \cite{La}, CMC 1 surfaces (i.e., surfaces with $H=1$) in hyperbolic 3-space  are (locally) in one-to-one correspondence with minimal surfaces in $\mathbb R^3.$ The latter can be explicitly parametrized in terms of holomorphic data
 by the classical Weierstrass representation.
 
 \begin{figure}[h]
\includegraphics[scale=0.25]{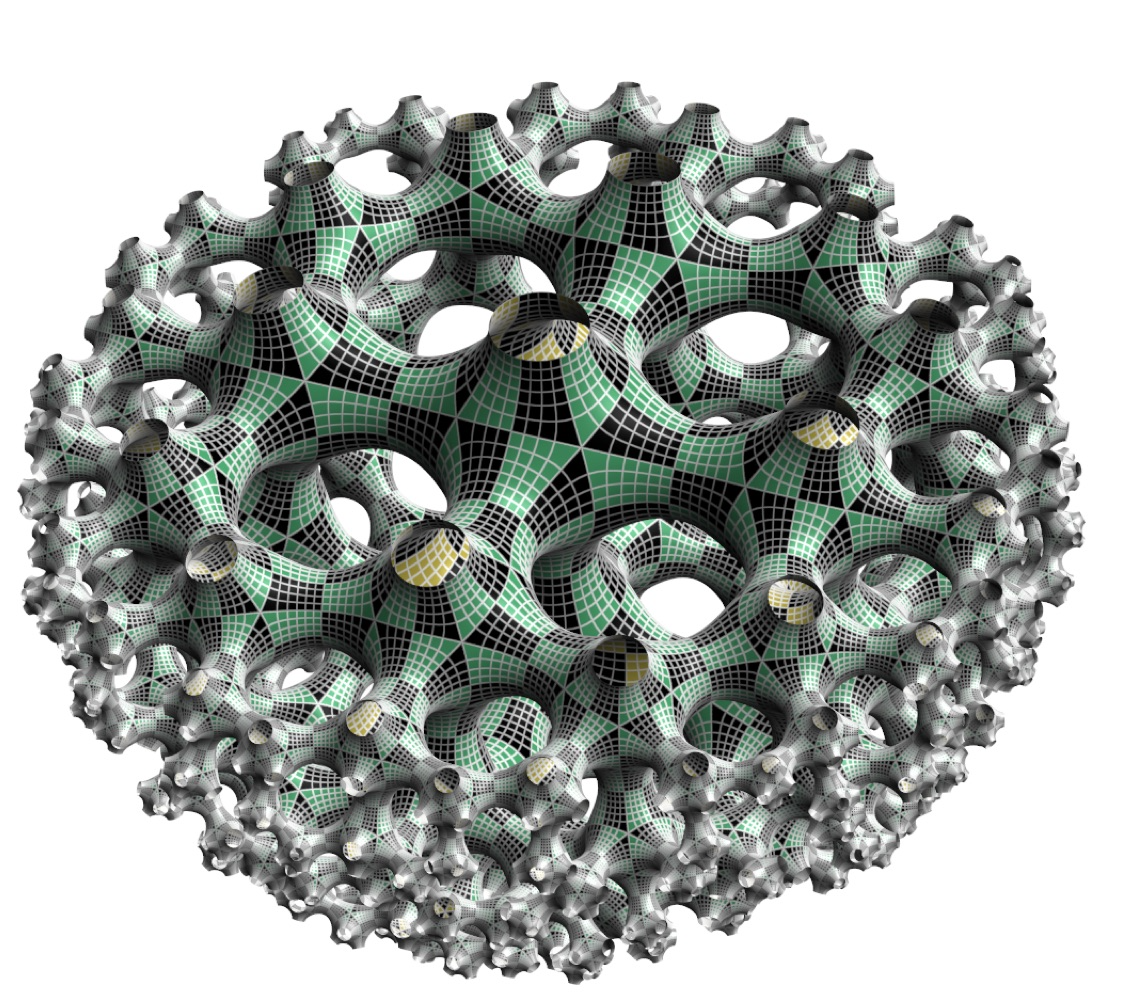}
\caption{An equivariant CMC 1 surface compatible with the dodecahedron tesselation of $\mathbb H^3$. The shown lines are the trajectories 
of a holomorphic quadratic differential of the underlying compact Riemann surface of genus 3.
Image by Nick Schmitt.}\label{fig:Tetsurface}
\end{figure}
 
As first noticed by Bryant \cite{Bry}, see also \cite{UY}, there exists a local Weierstrass representation of CMC 1 surfaces in $\mathbb H^3$ as well.
Consider for this a nilpotent nowhere vanishing $\mathfrak{sl}(2,\mathbb C)$-valued  holomorphic 1-form $\Phi$ on a Riemann surface $\Sigma.$ The frame $F$ is a solution of the ordinary differential equation  $dF=\Phi F$. Then the conformal immersion $f$ (defined on the universal covering of $\Sigma$) given by 
$f=\bar F^TF$  has constant mean curvature $1$ in the hyperbolic 3-space. In invariant terms, every CMC 1 surface is given by a flat unitary connection  $\nabla$ together with a nilpotent Higgs field $\Phi,$ see \cite{Pir}.  The frame $F$ is then a parallel frame with respect to the flat connection $\nabla-\Phi.$
The unitary connection $\nabla$ gives rise to an associated flat unitary bundle $V \rightarrow \Sigma$ whose transition functions are determined by the local unitary frames of $\nabla.$ Since the representation formula is not sensitive with respect to appropriate changes of the parallel unitary frames, the choice of the unitary frame does not alter the CMC 1 immersion.
We say a CMC 1 surface in a hyperbolic manifold admits a {\em  simple Weierstrass representation}, if the corresponding flat unitary bundle is trivial. The following theorem holds.
\begin{theorem}
A CMC 1 surface $f\colon \Sigma\to\mathbb H^3/\Gamma$ into a hyperbolic 3-manifold $\mathbb H^3/\Gamma$ admits a simple Weierstrass representation if and only if it is the projection of a holomorphic curve
$F\colon\Sigma\to\mathrm{SL}(2,\mathbb C)/\Gamma.$
\end{theorem}
Even though 
 we have shown the existence of holomorphic curves into $\mathrm{SL}(2,\mathbb C)/\Gamma$ where 
$\Gamma$ is the dodecahedron tesselation group, and  there exists compact CMC 1 surfaces in $\mathbb H^3/\Gamma$, see Figure \ref{fig:Tetsurface}, it is unclear whether there exist  compact  CMC 1 surfaces in $\mathbb H^3/\Gamma$ with simple Weierstrass representation, since the Higgs fields of the holomorphic curves we constructed here are not nilpotent.

\end{document}